\documentclass[a4paper,11pt]{article}

\usepackage{amsmath,amssymb,mathtools,amsthm}
\usepackage{xcolor,graphicx,float,enumerate}
\graphicspath{{figures/}}
\usepackage[title]{appendix}
\usepackage{tocloft}

\usepackage{empheq}
\usepackage{mathtools}
\usepackage{esint}

\usepackage[T1]{fontenc}
\usepackage[utf8]{inputenc}
\DeclareUnicodeCharacter{00A0}{ }
\usepackage{geometry}
\usepackage[hidelinks]{hyperref}
\usepackage[nameinlink]{cleveref}
\usepackage{enumerate}

\hypersetup{urlcolor=black, colorlinks=true}

\newcommand{\eps}{\epsilon}

\DeclareMathOperator{\N}{{\mathbb N}}

\newcommand{\ren}{{\mathbb{R}^n}}

\newcommand{\ifl}{\Delta_\infty^s}
\newcommand{\iflp}{\Delta_\infty^{s,+}}
\newcommand{\iflm}{\Delta_\infty^{s,-}}
\newcommand{\ifleps}{\mathcal{L}_\veps}
\newcommand{\fl}{(-\partial_{rr}^2)^s}
\newcommand{\flx}{(-\partial_{xx}^2)^s}

\newcommand{\beq}{\begin{equation}}
\newcommand{\eeq}{\end{equation}}

\DeclareMathOperator{\dist}{dist}

\newcommand*\diff{\mathop{}\!\mathrm{d}}

\newcommand{\vertiii}[1]{{\left\vert\kern-0.25ex\left\vert\kern-0.25ex\left\vert #1
		\right\vert\kern-0.25ex\right\vert\kern-0.25ex\right\vert}}
\newtheorem{theorem}{Theorem}[section]
\newtheorem{proposition}[theorem]{Proposition}%[section]
\newtheorem{corollary}[theorem]{Corollary}%[section]
\newtheorem{lemma}[theorem]{Lemma}%[section]

\theoremstyle{definition}
\newtheorem{definition}[theorem]{Definition}%[section]
\newtheorem{remark}[theorem]{Remark}%[section]

\numberwithin{equation}{section}

\newcommand{\indik}{\mathbf{1}}
\newcommand{\R}{\mathbb{R}}

\newcommand{\veps}{\varepsilon}
\newcommand{\dd}[1]{\diff{#1}}

\usepackage[
url=false,
isbn=false,
backend=bibtex,
maxcitenames=5,
giveninits,
maxbibnames=100,
doi=false,
eprint=false,
uniquename=false,
block=none,
sorting=nyt]{biblatex}
\renewbibmacro{in:}{}
\DeclareFieldFormat[article]{citetitle}{#1}
\DeclareFieldFormat[article]{title}{#1}

\makeatletter
\renewcommand*{\@fnsymbol}[1]{\ensuremath{\ifcase#1\or \star \or \dagger\or \ddagger\or
		\mathsection\or \mathparagraph\or \|\or **\or \dagger\dagger
		\or \ddagger\ddagger \else\@ctrerr\fi}}
\makeatother

\AtBeginBibliography{%
}

\definecolor{darkblue}{rgb}{0.05, .05, .65}
\definecolor{darkgreen}{rgb}{0.1, .65, .1}
\definecolor{darkred}{rgb}{0.8,0,0}

\title{Evolution Driven by the Infinity Fractional Laplacian}

\author{Félix del Teso\thanks{Departamento de Matemáticas, Universidad Autónoma de Madrid. \href{mailto:felix.delteso@uam.es}{felix.delteso@uam.es}}
	\and %
	J{\o}rgen Endal\thanks{Department of Mathematical Sciences, Norwegian University of Science and Technology. \href{mailto:jorgen.endal@ntnu.es}{jorgen.endal@ntnu.no}}
	\and %
	Espen R. Jakobsen\thanks{Department of Mathematical Sciences, Norwegian University of Science and Technology.
\href{mailto:espen.jakobsen@ntnu.no}{espen.jakobsen@ntnu.no}}
	\and %
	Juan Luis Vázquez\thanks{Departamento de Matemáticas, Universidad Autónoma de Madrid. \href{mailto:juanluis.vazquez@uam.es}{juanluis.vazquez@uam.es}}}

\begin{document}
	\maketitle
\begin{abstract}
We consider the evolution problem associated to the infinity fractional  Laplacian introduced by Bjorland, Caffarelli and Figalli (2012) as the infinitesimal generator of a non-Brownian tug-of-war game. We first construct a class of viscosity solutions of the initial-value problem  for bounded and uniformly continuous data.  An important result is the equivalence of the nonlinear operator in higher dimensions with the one-dimensional fractional Laplacian when it is applied to radially symmetric and monotone functions. Thanks to this and a comparison theorem between classical and viscosity solutions, we are able to establish a global Harnack inequality that, in particular, explains the long-time behavior of the solutions. 

\end{abstract}

\


\medskip


	\setcounter{tocdepth}{1}
	\tableofcontents
	
	\
%%%%%%%%%%%%%%%%%%%%%%%%%%%%%%%%%%%%%%%%%%%%%%%%%%%%%%%%%%%%%

\section{Introduction}\label{sec:intro}
In this paper we study a parabolic equation associated to the 
(normalized)  infinity fractional  Laplacian operator.  We recall that the local version of the game had been introduced by Peres et al. in 2009 (\cite{Pe-etal09}) where it is shown that the standard  infinity Laplace equation is solved by the value function for a  two-players  random turn “tug–of–war” game. The game is as follows: a token is initially placed at a position $x_0\in \Omega$ and every turn a fair coin is tossed to choose which of the players plays. This player moves the token to any point in the ball of radius $\veps>0$ around the current position. If, eventually, iterating this process, the token reaches a point $x_e\in \partial \Omega$, the players are awarded (or penalized) $f(x_e)$ (payoff function). For a PDE overview of the  infinity Laplacian operator and its role as an absolute minimizer for the $L^\infty$ norm of the gradient, see \cite{Lindqvist06, Lin16}.

 In 2012 Bjorland, Caffarelli and Figalli (\cite{BjCaFi12}) introduced equations involving the so-called \sl infinity fractional  Laplacian \rm as a model for a nonlocal version of the ``tug-of-war'' game. Following their explantation, instead of flipping a coin at every step,  every player chooses a direction and it is an $s$-stable L\'evy process that chooses both the active player and the distance to travel. The  infinity fractional Laplacian, with symbol $\ifl$, is a nonlinear integro-differential operator, the original definition is given in Lemma \ref{lem:EquivalentDefIFL} below. However, for the purpose of this paper, we also consider the alternative equivalent definition introduced in \cite{BjCaFi12} (see also \cite{DTEnLe22}) given by
\begin{equation}\label{eq:def1}
\ifl\phi(x):= C_s \sup_{|y|=1}\inf_{|\tilde{y}|=1} \int_{0}^\infty\left(\phi(x+\eta y)+\phi(x-\eta \tilde{y})-2\phi(x)\right) \frac{\dd \eta }{\eta ^{1+2s}}\qquad\text{where $s\in(1/2,1)$.}
\end{equation}
The constant is usually taken as  $C_s=(4^ss\Gamma(\frac{1}{2}+s))/(\pi^{\frac{1}{2}}\Gamma(1-s))$ but the value is irrelevant for our results. In their paper \cite{BjCaFi12} the authors study two stationary problems involving the infinity fractional  Laplacian posed in bounded space domains, namely, a Dirichlet problem and a double-obstacle problem.

Here, we consider the evolution problem
\begin{empheq}[left=\empheqlbrace]{align}
\partial_t u(x,t)&= \ifl u(x,t), \hspace{-3cm} &x \in \ren& , t > 0 ,\label{eq:EQ}\\
u (x,0) &= u_0(x),  \hspace{-3cm} & x \in \ren&, \label{eq:BC}
\end{empheq}
with  $s\in(1/2,1)$ and $n\ge 2$. When
$n=1$ the operator $- \ifl$ is just the usual linear fractional
Laplacian operator $(-\Delta)^s$ of order $s$, and equation
\eqref{eq:EQ} is just the well-known fractional heat equation \cite{BlumGet1960, Kol11}. See also a
detailed study of that equation using PDE techniques in \cite{DrGaVo03, BaPeSoVa14, BoSiVa17, Vazquez2018}.   Note that for $n\ge 2$ the operator is nonlinear so a
new theory is needed. A non-normalized version is introduced in
\cite{ChaJa17} along with a well-posedness theory for the
corresponding equations of the type \eqref{eq:EQ}--\eqref{eq:BC}. However, the two problems
are not equivalent nor closely related. 

Here we develop an existence theory of suitable viscosity solutions
for the parabolic problem \eqref{eq:EQ}--\eqref{eq:BC}, based on
approximation with monotone schemes. We show that the obtained class
of solutions enjoys a number of good properties.  As in the elliptic
case \cite{BjCaFi12}, we lack a uniqueness result in the context of
viscosity solutions. However, we are able to prove an important
comparison theorem relating two types of solutions, classical and
viscosity ones, see Theorem \ref{thm:comparisonsmooth}. As a counterpart, we also obtain uniqueness and comparison of classical solutions. Moreover, we show that for smooth, radially symmetric
functions and nonincreasing along the radius in $\ren$ with $n\geq2$, the operator $-\ifl$
reduces to the classical fractional Laplacian $(-\Delta)^s$ in
dimension $n=1$ (Theorem \ref{prop:radialOp}). A similar example regarding nondecreasing one-dimensional profiles can be found in Lemma \ref{lem:OtherExamplesSmoothSolutions}. In this way we may
construct a large class of classical solutions that
make the comparison theorem relevant (Theorems \ref{thm:ExistenceClassicalSolution} and \ref{thm:ExistenceClassicalSolution2}).  Note that no similar reduction
applies in general, even in the radial case (see
Subsection \ref{sec:radialOpCounter} for a counterexample).

 Using the developed tools, we study the asymptotic behavior of the constructed solutions, and obtain a global Harnack type principle, see Theorem \ref{thm:GlobalHarnackPrinciple}. 
 
 \subsection{Related literature}

It is interesting to compare the nonlocal model \eqref{eq:EQ} with the local version of the infinity Laplacian that has been studied by many authors, both in the stationary and evolution cases, cf.   \cite{AronssonCrJuut2004,JuutinenPK2006,Lindqvist06, AkJuKa09, Pe-etal09, MR2915863, PortVaz2013, Lin16}.  
Asymptotic expansions for the game theoretical  $p$-Laplacian in the local case and related approximation schemes in the elliptic case are studied in \cite{Manfredi10, Manfedi12, dTMaPa22} and in the parabolic case in \cite{MaPaRo10}. For the variational version of the $p$-Laplacian operator see \cite{dTLi21}.

There exist in the literature other nonlocal generalizations of the $p$-Laplacian and the infinity Laplacian. Let us mention (i) the normalized version \cite{BjCaFi12, BCF12b} with asymptotic expansions and game theoretical approach \cite{BuSq22, DTEnLe22, Lew22};  (ii) nonnormalized version \cite{ChaJa17} both elliptic and parabolic; (iii) H\"older infinity Laplacian \cite{ChLiMo12}; and (iv) the (variational) fractional $p$-Laplacian \cite{CoHa16, MaRoTo16, Vaz16, Pala18,Vaz20, CoHa21, Vaz21}.

\section{Preliminaries and statement of main results}		

%%%%%%%%%%%%%%%%%%%%%%%%%%%%%%%%%%%%%%%%%%%%%%%%%%%%%%%%%%%
First let us fix some notation that we will use along the paper.

For given $\delta>0$, standard mollifiers are denoted by $\rho_\delta$. Following \cite{BjCaFi12}, we say that $\phi\in C^{1,1}(x)$ at some $x\in \ren$ if there exists $p_x\in \ren$ and $C_x,\eta_x>0$ such that
\begin{equation}\label{eq:PhiTaylor}
|\phi(x+y)-\phi(x)-p_x\cdot y| \leq C_x |y|^2\qquad\text{for all $|y|<\eta_x$.}
\end{equation}
Note that $C_\textup{b}^2(B_R(x))\subset C^{1,1}(x)$.  Here
$C_{\textup{b}}^k(U)$ is the space of functions on the set $U$ with bounded continuous
derivatives of all orders in $[0,k]$. Let us also define:
\begin{align*}
B(\ren)&:=\{\phi: \ren\to\R \,|\, \textup{$\phi$ is pointwisely bounded}\},\\
UC(\ren)&:=\{\phi: \ren\to\R \,|\, \textup{$\phi$ is uniformly continuous}\},\\
BUC(\ren)&:=B(\ren)\cap UC(\ren)\qquad\textup{with}\qquad
\|\phi\|_{C_\textup{b}(\ren)}:=\sup_{x\in\ren}|\phi(x)|,\\
\intertext{ and for $\beta\in(0,1]$, we define
$|\phi|_{_{C^{0,\beta}}(\ren)}=\sup_{x,y\in\ren}|\phi(x)-\phi(y)|/|x-y|^\beta$
and} 
 C^{0,\beta}(\ren)&:=\{\phi\in
C_{\textup{b}}(\ren) \,|\, \textup{$\|\phi\|_{C^{0,\beta}}<\infty$}\} \qquad\text{where}\qquad  \|\phi\|_{C^{0,\beta}}=\|\phi\|_{C_{\textup{b}}}+|\phi|_{C^{0,\beta}}.
\end{align*}
 A modulus of continuity is a nondecreasing function $\omega:\R^+\to \R^+$ such that
$\lim_{r\to 0^+}\omega(r)=0$. For a function $f\in BUC(\ren)$, we define the corresponding modulus of
continuity as follows:
$$\omega_f(r)= \sup_{|y|\leq
  r}\|f(\cdot+y)-f\|_{C_\textup{b}(\ren)}.$$
For a H\"older continuous function $f\in C^{0,\beta}(\ren)$,
$\omega_f(r)\leq |f|_{C^{0,\beta}}r^\beta$.

We will also need $e_i:=(0,0,\ldots,0,1,0,\ldots,0)\in\ren$, where $1$ is at the $i$th component.

\subsection{Alternative characterization of the infinity fractional Laplacian}

We have the following alternative characterization of operator $\ifl$ that we will use throughout:

\begin{lemma}[Alternative characterization]\label{lem:EquivalentDefIFL}
Assume $\phi \in C^{1,1}(x)\cap B(\ren)$. Then:
\begin{enumerate}[$\bullet$]
\item If $\nabla\phi(x)\not=0$, then
\begin{equation*}%\label{eq:def2a}
\ifl \phi(x)=C_s\int_0^\infty\big(\phi\left(x+ \eta  \zeta\right)+\phi\left(x- \eta  \zeta\right)-2\phi(x)\big)\frac{\dd  \eta }{ \eta^{1+2s}}\qquad\text{where $\zeta:=\nabla\phi(x)/|\nabla\phi(x)|$.}
\end{equation*}
\item If $\nabla\phi(x)=0$, then
\begin{equation*}%\label{eq:def2b}
\ifl \phi(x)=C_s\sup_{|y|=1} \int_{0}^\infty\big(\phi(x+ \eta y)-\phi(x)\big) \frac{\dd \eta}{\eta^{1+2s}} + C_s\inf_{|y|=1} \int_{0}^\infty\big(\phi(x-\eta y)-\phi(x)\big) \frac{\dd \eta}{\eta^{1+2s}}.
\end{equation*}
\end{enumerate}
\end{lemma}

The equivalence when $\nabla \phi(x)=0$ follows from the fact that the
integrals in this case are well-defined and can be combined to
get \eqref{eq:def1}. When $\nabla\phi(x)\neq0$, it can be shown that the supremum and infimum of \eqref{eq:def1} is actually taken at $\zeta$, see Proposition 2.2 in \cite{DTEnLe22}. To sketch the proof, assume for simplicity that the supremum in \eqref{eq:def1} is taken at $y$, and let us argue that $y=\zeta$. Indeed, by splitting the integral and using the definitions of $C^{1,1}$ and the infimum,
$$
\ifl \phi(x)\leq C_s\int_0^\infty\big(\phi\left(x+\eta y\right)+\phi\left(x-\eta \zeta\right)-2\phi(x)\big)\frac{\dd \eta}{\eta^{1+2s}}\leq C_s\big(\nabla\phi(x)\cdot(y-\zeta)\big)\int_0^{\eta_x}\eta \frac{\dd \eta}{\eta^{1+2s}}+C.
$$
Now, since $\ifl \phi(x)$ is well-defined and the integral diverges if $y\not=\zeta$, we must have $y=\zeta$. A similar argument holds for the infimum.

\subsection{Existence of solutions and basic properties}
We are able to construct a suitable class of viscosity solutions of \eqref{eq:EQ}--\eqref{eq:BC}. The two steps are as follows:\\
(i) Approximating $\ifl$ by removing the singularity, i.e., we introduce
\begin{equation*}
\ifleps[\phi](x):=C_s \sup_{|y|=1}\inf_{|\tilde{y}|=1}\int_\veps^\infty\big(\phi(x+\eta y)+\phi(x-\eta \tilde{y})-2\phi(x)\big)\frac{\dd \eta }{\eta ^{1+2s}}.
\end{equation*}
(ii) Discretizing in time by letting $\tau>0$ and $t_j:=j\tau$ for $j\in\N$, i.e., $t_j\in \tau\N$, and then considering the semidiscrete problem
\begin{empheq}[left=\empheqlbrace]{align}
\frac{U^{j+1}(x)-U^j(x)}{\tau}&=\ifleps[U^j](x), \hspace{-2cm} &x \in \ren& , j\in \N  ,\label{eq:EQD}\\
U^0 (x) &= u_0(x),  \hspace{-2cm} & x \in \ren& . \label{eq:BCD}
\end{empheq}
We study the properties of \eqref{eq:EQD}--\eqref{eq:BCD} in Section \ref{sec:scheme}.
Existence of viscosity solutions follows by taking the limit in this approximate scheme, as well as properties inherited from the approximations.

\begin{theorem}[Existence and a priori results]\label{thm:ExistenceAPrioriViscosity}
 If $u_0\in BUC(\ren)$, then there is at least one viscosity solution $u\in C_{\textup{b}}(\ren\times[0,\infty))$ of \eqref{eq:EQ}--\eqref{eq:BC}. Moreover:
\begin{enumerate}[\rm (a)]
\item\label{thm:ExistenceAPrioriViscosity-item1} \textup{($C_\textup{b}$-bound)} For all $t>0$, $\|u(\cdot,t)\|_{C_\textup{b}(\ren)}\leq \|u_0\|_{C_\textup{b}(\ren)}$.
\item\label{thm:ExistenceAPrioriViscosity-item2} \textup{(Uniform continuity in space)} For all $y\in\ren$ and all $t>0$,
\item[]\quad$\|u(\cdot+y,t)-u(\cdot,t)\|_{C_\textup{b}(\ren)}
  %\leq \|u_0(\cdot+y)-u_0\|_{C_\textup{b}(\ren)}
  \leq \omega_{u_0}(|y|).$\smallskip
\item\label{thm:ExistenceAPrioriViscosity-item3}
  \textup{(Uniform continuity in time)} For all $t,\tilde{t}>0$,%\\[0.3cm]
\item[]\quad$\|u(\cdot,t)-u(\cdot,\tilde{t})\|_{C_\textup{b}(\ren)}\leq \tilde
  \omega(|t-\tilde t|)\qquad\text{where $\tilde
  \omega(r):=\inf_{\delta>0}\Big\{\omega_{u_0}(\delta)+r\sup_{\veps>0}\|\ifleps[u_{0,\delta}]\|_{C_\textup{b}(\ren)}\Big\}$}$
  \item[] is a modulus satisfying
    $\tilde\omega(r)\leq \omega_{u_0}(r^{1/3})
+C\big(r^{1/3}+r\big)$, $C:=c_s\|u_0\|_{C_\textup{b}(\ren)}\|\nabla\rho\|^{2-2s}_{L^1(\ren)}\|D^2\rho\|_{L^1(\ren)}^{2s-1}$, and $\rho$ is a standard mollifier.
%\max\{\frac{1}{2(1-s)}\|\nabla^2\rho\|_{L^1(\ren)}, \frac{2}{s}\}.$\nc
\end{enumerate}
\end{theorem}

\begin{remark}
  The 
  definition of viscosity solutions is given %can be found in
  Section \ref{sec:existence} (Definition \ref{def:viscsol}).
  We obtain viscosity solution as limits of monotone approximations of
  the problem in Section \ref{sec:scheme}. 
  \end{remark}

Note that if $u_0$ is H\"older continuous and $s\in(1/2,1)$, then the above modulii will
be (more) explicit.

\begin{lemma}\label{lem:Holder}
If $u_0\in C^{0,\beta}(\ren)$ for $\beta\in(0,1]$, then
    $$\omega_{u_0}(\delta)=|u_0|_{C^{0,\beta}}\delta^\beta\qquad\text{and}\qquad
    %\tilde\omega(r)=
    \|\ifleps[u_{0,\delta}]\|_{C_\textup{b}(\ren)}\leq c(s,\rho) |u_0|_{C^{0,\beta}}  
%c_{s}\|Du_{0,\delta}\|_{C_\textup{b}(\ren)}^{2-2s}\|D^2u_{0,\delta}\|_{C_\textup{b}(\ren)}^{2s-1}\leq 
%\|D\rho\|_{C_\textup{b}(\ren)}^{2-2s}\|D^2\rho\|_{C_\textup{b}(\ren)}^{2s-1}\|u_0\|_{C^{0,\beta}}
\delta^{\beta-2s}.
    $$
  \end{lemma}
The above result will be proved at the end of Section \ref{sec:existence}.

It follows after a minimization in $\delta$ that
$\tilde\omega(r)=c(s,\rho)|u_0|_{C^{0,\beta}}r^{\frac 1{2s}}$, and the
solution $u$ will be H\"older continuous with the correct parabolic regularity.

\begin{corollary}[Existence and a priori results]\label{thm:ExistenceAPrioriViscosity_cor}
If $u_0\in C^{0,\beta}(\ren)$ for $\beta\in(0,1]$,
  then there is at least one viscosity solution $u\in C_{\textup{b}}(\ren\times[0,\infty))$ of \eqref{eq:EQ}--\eqref{eq:BC}. Moreover:
\begin{enumerate}[\rm (a)]
\item\label{thm:ExistenceAPrioriViscosity-item1} \textup{($C_\textup{b}$-bound)} For all $t>0$ $\|u(\cdot,t)\|_{C_\textup{b}(\ren)}\leq \|u_0\|_{C_\textup{b}(\ren)}$.
\item\label{thm:ExistenceAPrioriViscosity-item2} \textup{(H\"older in
  space)} For all $y\in\ren$ and all $t>0$,
\item[]\quad$\|u(\cdot+y,t)-u(\cdot,t)\|_{C_\textup{b}(\ren)}\leq
|u_0|_{C^{0,\beta}}|y|^\beta.$\smallskip
\item\label{thm:ExistenceH\"older in time)} \textup{(H\"older in time)} There is a constant $c(s,\rho)$ only depending on $s$ and
  $\rho$ such that for all $t,\tilde{t}>0$,%\\[0.3cm]
\item[]\quad$\|u(\cdot,t)-u(\cdot,\tilde{t})\|_{C_\textup{b}(\ren)}\leq
  C |u_0|_{C^{0,\beta}}|t-\tilde t|^{\frac\beta{2s}}$.
\end{enumerate}
\end{corollary}

\subsection{Classical solutions, radial solutions, comparison, and uniqueness}
There could be other ways of obtaining viscosity solutions, and
unfortunately, we lack general comparison and uniqueness
results. Nevertheless, we can obtain that classical solutions are unique and we can compare our constructed viscosity solutions with classical sub- and supersolutions of \eqref{eq:EQ}--\eqref{eq:BC}.\footnote{We will work with classical solutions in $C_{\textup{b}}^2$. Actually, we can reduce to $C_{\textup{b}}^1$ for the temporal variable, and to $C^{1,1}\cap B$ for the spatial variables.}

\begin{theorem}[Comparison between classical and viscosity solutions]\label{thm:comparisonsmooth}
Assume $u_0\in BUC(\ren)$. Let  $\underline{u},\overline{u}\in C^2_\textup{b}(\ren\times[0,\infty))$ be respective 
  classical sub- and supersolution of \eqref{eq:EQ}--\eqref{eq:BC}, and let
  $u\in BUC(\ren\times[0,\infty)$ be a
    viscosity solution of \eqref{eq:EQ}--\eqref{eq:BC} as
    constructed in Theorem \ref{thm:ExistenceAPrioriViscosity}. Then $ \underline{u} \leq u \leq \overline{u}$ in
    $\ren\times(0,\infty)$.
\end{theorem}

The above result is proved in Section \ref{sec:compsmooth}. We want to emphasize that it is done in a rather nonstandard way, since we inherit the comparison from the approximation scheme when the solution is classical. In general, this cannot be done in the context of viscosity solutions since the approximation scheme only converges up to a subsequence. 

\begin{remark}
By Theorem \ref{thm:comparisonsmooth}, we can in addition get comparison of constructed viscosity solutions as long as the initial datas are separated by an initial data which produces a classical solution.
\end{remark}

An immediate consequence of Theorem \ref{thm:comparisonsmooth}:

\begin{corollary}[Comparison of classical sub- and supersolutions]
Let $u,v\in C^2_\textup{b}(\ren\times[0,\infty))$ be respective classical sub- and supersolutions of \eqref{eq:EQ}--\eqref{eq:BC} with initial data $u_0,v_0$. If $u_0\leq v_0$, then $u\leq v$.
\end{corollary} 

\begin{corollary}[Uniqueness of solutions]
Classical solutions of \eqref{eq:EQ}--\eqref{eq:BC} in $C^2_\textup{b}(\ren\times[0,\infty))$ are unique. 
\end{corollary}

Theorem \ref{thm:comparisonsmooth} might be an empty statement unless we provide a class of classical solutions of \eqref{eq:EQ}--\eqref{eq:BC}. The following result, proved in Section \ref{sec:radialsol}, solves this issue.

\begin{theorem}[Existence of  classical  radial  solutions]\label{thm:ExistenceClassicalSolution}
Assume that $u_0\in C_\textup{b}^\infty(\ren)$ is radial and
radially nonincreasing. Then there exists a classical and radial
 solution $u\in C^\infty_\textup{b}(\ren\times[0,\infty))$ of
  \eqref{eq:EQ}--\eqref{eq:BC}. 
   Moreover, if $U_0(r):=u_0(|x|)$ and $U_0(-r):=U_0(r)$ for $r=|x|\geq 0$, then
\[
u(x,t)= (P_{s}(\cdot,t)\ast U_0)(r)= \int_{-\infty}^\infty P_{s}(r-s,t) U_0(s)\dd s \qquad \textup{for all $|x|=r$},
\]
where  $P_{s}$ is the fundamental solution of the one-dimensional fractional heat equation (cf. \eqref{fht.profile}).
\end{theorem}

\begin{remark}
\begin{enumerate}[(a)]
\item The idea in the above result is that, for radially nonincreasing radial functions, the
  operators $-\Delta_{\infty,\R^n}^s$ %in any dimension
  and $(-\Delta)^s_{\R^1}$ %in dimension $n=1$
  coincide (Proposition \ref{prop:radialOp}), and \eqref{eq:EQ} then reduces to the one-dimensional
  fractional heat equation. % in dimension $n=1$.
%  We provide such a result in . Note that the fractional heat equation has classical solutions even for merely bounded initial data \cite{DrGaVo03,BaPeSoVa14,BoSiVa17}.
\item In view of Theorem \ref{thm:comparisonsmooth}, this classical solution is also a viscosity solution in our sense.
\end{enumerate}
\end{remark}

Another class of classical solutions are:

\begin{theorem}[Existence of  classical solutions with one-dimensional profiles]\label{thm:ExistenceClassicalSolution2}
Assume that $U_0\in C^\infty_{\textup{b}}(\R)$ is nondecreasing, and let $u_0\in C^\infty_{\textup{b}}(\ren)$ be defined as
\[
u_0(x):=U_0(x_1).
\] 
Then there exists a classical solution $u\in C^\infty_\textup{b}(\ren\times[0,\infty))$ of
  \eqref{eq:EQ}--\eqref{eq:BC}. 
   Moreover, 
\[
u(x,t)= (P_{s}(\cdot,t)\ast U_0)(x_1)= \int_{-\infty}^\infty P_{s}(x_1-s,t) U_0(s)\dd s,
\]
where  $P_{s}$ is the fundamental solution of the one-dimensional fractional heat equation (cf. \eqref{fht.profile}).
\end{theorem}

The proof is similar to the one of Theorem \ref{thm:ExistenceClassicalSolution}, and we will omit it. One just needs to note that $P_{s}(\cdot,t)\ast U_0$ is nondecreasing.

\subsection{Asymptotic behavior and Harnack inequality}
Having established Theorems \ref{thm:comparisonsmooth} and \ref{thm:ExistenceClassicalSolution}, we can prove that solutions of \eqref{eq:EQ}--\eqref{eq:BC} behave like solutions of the one-dimensional fractional heat equation, up to suitable constants. In Section \ref{sec:heat}, we recall some results on that equation and its fundamental solution denoted by $P_s$. In Section \ref{sec:GHP}, we prove the following result.

\begin{theorem}[Global Harnack principle]\label{thm:GlobalHarnackPrinciple}
Let $u\in BUC(\ren\times[0,\infty))$ be a viscosity solution of \eqref{eq:EQ}--\eqref{eq:BC}, as constructed in Theorem \ref{thm:ExistenceAPrioriViscosity}, with initial data $u_0\in BUC(\ren)$ such that  $u_0\not\equiv 0$ and
\[
0\leq u_0(x)\leq (1+|x|^2)^{-\frac{1+2s}{2}} \qquad \text{for all $|x|\geq R\geq 1$.}
\]
Then, for all $\tau>0$, there exist constants $C_1,C_2>0$ depending only on $s$, $R$, and
$u_0$, such that
\begin{equation*}%\label{E:cauchy-bounds}
C_1 P_s(|x|,t) \leq u(x,t) \leq C_2 P_s(|x|,t)\qquad\text{for all $(x,t)\in\ren\times[\tau,\infty)$.}
\end{equation*}
Moreover, for all $\tau>0$, there exist constants $\tilde{C}_1,\tilde{C}_2>0$ depending only on $s$, $R$, and $u_0$, such that 
\[
\tilde{C}_1\frac{t}{(t^{\frac{1}{s}}+|x|^2)^{\frac{1+2s}{2}}} \leq u(x,t) \leq \tilde{C}_2\frac{t}{(t^{\frac{1}{s}}+|x|^2)^{\frac{1+2s}{2}}}\qquad\text{for all $(x,t)\in\ren\times[\tau,\infty)$.}
\]
\normalcolor In particular, $u>0$ in $\ren\times[\tau,\infty)$.
\end{theorem}

\begin{remark}
\begin{enumerate}[(a)]
\item Note that $u_0$ is not necessarily in $L^1(\ren)$ since the decay required for large $x$ is the one corresponding to the one-dimensional fractional heat kernel $P_s$.
\item The above theorem provides a counterexample to conservation of
  mass for \eqref{eq:EQ}--\eqref{eq:BC}: For any smooth compactly supported $0\leq u_0\in L^1(\ren)$, the corresponding solution $u$ satisfies
\[
\int_{\ren} u(x,1) \dd x \geq \tilde{C}_1 \int_{\ren}\frac{1}{(1+|x|^2)^{\frac{1+2s}{2}}}  \dd x.
\]
The last integral is infinite if $1+2s\leq n$, 
  and hence there is no conservation of mass for $n \geq 3$.
\item In Theorem \ref{thm:ExistenceClassicalSolution2}, we construct other types of special solutions which could also be used to prove the global Harnack principle.
\end{enumerate}
\end{remark}

\section{Properties of a approximation scheme}\label{sec:scheme}

We will now start the detailed development of the theory. The basic idea we follow is to discretize explicitly in time and use the asymptotic expansion of $\ifl$ found in \cite{DTEnLe22} to provide a monotone zero-order approximation of the operator.

We recall that, for $s\in(1/2,1)$,
\begin{equation}\label{eq:defLeps}
\begin{split}
\ifleps[\phi](x)&:=C_s \frac{1}{2s \veps^{2s}}\left(\sup_{|y|=1}\fint_\veps^\infty \phi(x+\eta y)  \frac{\dd \eta }{\eta^{1+2s}}  + \inf_{|y|=1}\fint_\veps^\infty \phi(x+\eta y)  \frac{\dd \eta }{\eta^{1+2s}}-2\phi(x)\right)\\
&= C_s \left(\sup_{|y|=1}\int_\veps^\infty \phi(x+  \eta y)  \frac{\dd \eta }{\eta ^{1+2s}}  + \inf_{|y|=1}\int_\veps^\infty \phi(x+\eta y)  \frac{\dd \eta }{\eta ^{1+2s}}-\frac{1}{s\veps^{2s}}\phi(x)\right)\\
&=C_s \sup_{|y|=1}\inf_{|\tilde{y}|=1}\int_\veps^\infty\big(\phi(x+\eta y)+\phi(x-\eta \tilde{y})-2\phi(x)\big)\frac{\dd \eta }{\eta ^{1+2s}}.
\end{split}
\end{equation}

\begin{lemma}\label{lem:selfmap}
The operator $\ifleps:C_{\textup{b}}(\ren)\to C_{\textup{b}}(\ren)$ is well-defined
and bounded. 
%\ifleps:C_\textup{b}(\ren)\to C_\textup{b}(\ren)
\end{lemma}

\begin{remark}
Note that, in general, $\ifl:C_{\textup{b}}^\infty(\ren) \not \to C_{\textup{b}}(\ren)$. See Section \ref{sec:extensions}.
\end{remark}

\begin{proof}[Proof of Lemma \ref{lem:selfmap}]
Let $\phi\in C_\textup{b}(\ren)$. Since $\int_{\veps}^\infty
\eta^{-(1+2s)} \dd t=\frac{1}{2s}\veps^{-2s}$, we have
$\|\mathcal L_{\veps}[\phi]\|_{C_\textup{b}(\ren)}\leq
\frac{4C_s}{2s}\|\phi\|_{C_\textup{b}(\ren)}\veps^{-2s}$ for any
$\veps>0$. It follows that $\ifleps[\phi]$ is bounded. If $\mathcal
L_{\veps}[\phi]$ is continuous, it also follows that $\ifleps$ is a
bounded operator on $C_{\textup{b}}(\ren)$.  To show
continuity at an arbitrary point $x_1\in\R^n$, we fix $\eps>0$. By the above bound there is (large) $R>0$ such that
\begin{align*}
\|\mathcal L_{R}[\phi]\|_{C_\textup{b}(\ren)} < \frac{\eps}{4}.
\end{align*}
For $x_2\in B(x_1,1)$, we find by the triangle inequality and $\sup\inf(\cdots)-\sup\inf(\cdots)\leq
\sup\sup(\cdots-\cdots)$,
\begin{equation*}
\begin{split}
  &\big|\ifleps[\phi](x_1)-\ifleps[\phi](x_2)\big|\\
  &\leq  C_s
\sup_{|\tilde y|=1}\sup_{|y|=1}\int_\veps^{R} \big|\big(\phi(x_1+\eta y)+
\phi(x_1+\eta \tilde y)\big) -  \big(\phi(x_2+\eta y)+\phi(x_2+\eta \tilde y)\big)\big| \frac{\dd \eta }{\eta ^{1+2s}}\\
&\quad+2C_s\int_\veps^{R}\frac{\dd \eta }{\eta ^{1+2s}}|\phi(x_1)-\phi(x_2)|+\frac\eps4 +\frac\eps4 \\
&\leq 2C_s \omega_{\phi,R}(x_1-x_2)\int_\veps^{R}\frac{\dd \eta }{\eta ^{1+2s}}+ \frac\eps2,
\end{split}
\end{equation*}
where $\omega_{\phi,R}$ is the modulus of continuity of $\phi$ in the
ball $B(0,|x_1|+1+R)$. Since the integral is finite,  the last expression
is less than $\veps$ when $|x_2- x_1|$ is small enough and continuity of $\ifleps[\phi]$
follows.
%it immediatly follows that there is $\delta\in(0,1)$, such that $|x_1-x_2|<\delta$ implies $|\ifleps[\phi](x_1)-\ifleps[\phi](x_2)\big|<\eps$ and continuity at $x_1$ follows.
  \end{proof}
 To state the consistency, we introduce admissible test functions
$\phi$: There is
$\eta_x>0$, such that
$$\textup{(i)} \quad \phi\in C^2(\bar B_{\eta_x})\qquad\text{and}\qquad \textup{(ii)} \quad \phi\in B(\ren)\cap UC(\ren\setminus \bar B_{\eta_x}).$$
\begin{lemma}[Consistency, Theorem 1.1 in \cite{DTEnLe22}]\label{lem:Consistency}
Under the above assumptions on $\phi$, for every $\veps< \eta_x$,
\begin{equation*}
\Big|\ifleps[\phi](x)-\ifl\phi(x) \Big|=o_\veps(1),
\end{equation*}
where the bound $o_\veps(1)$ depends only on 
$|\nabla\phi(x)|^{-1}$, $\|D^2\phi\|_{C_{\textup{b}}(\bar B_{\eta_x})}$, and
$\omega_{\phi,\bar B_{\eta_x}^c}$.
\end{lemma}

%% \begin{lemma}[Uniform bound, Lemma 2.1 in \cite{DTEnLe22}
%%   ]\label{lem:unifBoundLveps}
%% Under the above assumptions on $\phi$, we have that $\{\ifleps[\phi](x)\}_{\veps>0}$ is bounded independently of $\veps$. More precisely, for all $\veps<\eta_x$ there holds:
%% \begin{equation*}
%% |\ifleps[\phi](x)| \leq \frac{C_s}{(1-s)}C_x\eta_x^{2-2s}+\frac{2C_s}{s}\|\phi\|_{C_\textup{b}(\ren)}\eta_x^{-2s}.
%% \end{equation*}
%% \end{lemma}

We also need $\veps$ independent bounds to send $\veps\to0$.
\begin{lemma}[Uniform bound]\label{lem:unifBoundLveps}
  If $\phi\in C_{\textup{b}}^2(\ren)$, then there is a constant
  $c(s)$ only depending on $s$ such that
\begin{equation*}
|\ifleps[\phi](x)| \leq c(s)\|\nabla\phi\|_{ C_{\textup{b}}(\ren)}^{2-2s}\|D^2\phi\|_{ C_{\textup{b}}(\ren)}^{2s-1}.
\end{equation*}
\end{lemma}
\begin{proof}
We add to $\ifleps$ the gradient
term
$$
\sup_{|y|=1}\inf_{|\tilde{y}|=1} \int_\veps^{1} tp_x\cdot(y-\tilde{y})  \frac{\dd t}{t^{1+2s}}
= \int_\veps^{1} \eta  \frac{\dd \eta }{\eta ^{1+2s}}\sup_{|y|=1}\inf_{|\tilde{y}|=1}p_x\cdot(y-\tilde{y})=0,
$$
Since the intgrand is bounded, we then split the resulting integral in
two---an integral 
with the $\inf$ and an integral with the $\sup$. The result for the
$\sup$-part is:
\begin{align*}
\sup_{|y|=1}\int_\veps^\infty\big(\phi(x+\eta
y)-\phi(x)-tp_x\cdot y \eta 1_{0<\eta<1}\big)\frac{\dd \eta }{\eta
  ^{1+2s}}.
\end{align*}
Splitting this integral in two, $\int_\veps^r+\int_r^\infty$, and Taylor
expanding, we find the following upper bound
\begin{align*}
 \frac12\|D^2\phi\|_{C_{\textup{b}}}\int_0^r\eta^2 \frac{\dd \eta }{\eta
  ^{1+2s}} + 2\|\nabla\phi\|_{C_{\textup{b}}}\int_r^\infty\eta \frac{\dd \eta }{\eta
    ^{1+2s}}\leq \frac12\|D^2\phi\|_{C_{\textup{b}}}\frac1{2-2s}r^{2-2s}+ 2\|\nabla\phi\|_{C_{\textup{b}}}\frac1{2s-1}r^{1-2s}.
  \end{align*}
Minimizing with respect to $r$ then proves the result for the $\sup$-part.
The $\inf$-part is similar. 
\end{proof}

\begin{remark}\label{remark:schemeform}  Note that $\ifleps$ is monotone in the following two ways:
\begin{enumerate}[{\rm (i)}] 
\item $\ifleps[\phi]\leq 0$ at any global maximum of $\phi$.
\item In the sense of monotone approximations in viscosity solution theory:
\[\psi_1\leq \psi_2\qquad \text{in $\ren$} \qquad\implies\qquad L(\veps,\psi_1,r)\leq
L(\veps,\psi_2,r)\qquad \text{in $\ren$}, \]
where $\ifleps[\psi](x)=L(\veps, \psi, \psi(x))$ and
 $L:\R_+\times BUC(\ren)\times \R$ is given by
\[
L(\veps,\psi,r)= C_s \Big(\sup_{|y|=1}\int_\veps^\infty \psi(x+\eta y)  \frac{\dd \eta }{\eta ^{1+2s}}  + \inf_{|y|=1}\int_\veps^\infty \psi(x+\eta y)  \frac{\dd \eta }{\eta ^{1+2s}}-\frac{1}{s\veps^{2s}}r\Big).
\]
%If $\psi_1\leq \psi_2$ in $\ren$ then $L(\veps,\psi_1,r)\leq L(\veps,\psi_2,r)$.
\end{enumerate}
These properties are crucial in order to obtain approximation schemes that preserves the properties of the limit problem \eqref{eq:EQ}--\eqref{eq:BC}.
\end{remark}

%%%%%%%%%%%%%%%%%%%%%%%%%%%%%%%%%%%%%%%%%%%%%%%%%%%%%%%%%%%%%
\subsection{Semi-discrete scheme defined on $\ren\times\{\tau \N\,\cup\,0\}$}

We will now study the semi-discrete scheme \eqref{eq:EQD}--\eqref{eq:BCD}. 

%Let $\tau>0$ and $t_j:=j\tau$ for $j\in\N$, i.e., $t_j\in \tau\N$. We then consider the semidiscrete problem
%\begin{empheq}[left=\empheqlbrace]{align}
%\frac{U^{j+1}(x)-U^j(x)}{\tau}&=\ifleps[U^j](x), \hspace{-2cm} &x \in \ren& , j\in \N  ,\label{eq:EQD}\\
%U^0 (x) &= u_0(x),  \hspace{-2cm} & x \in \ren& . \label{eq:BCD}
%\end{empheq}

\begin{proposition}[Well-posedness and properties]\label{prop:propscheme}
 Assume $u_0\in BUC(\ren)$ and $\veps,\tau>0$. Then there
exists a unique  solution $U^j\in C_{\textup{b}}(\ren)$ of
\eqref{eq:EQD}--\eqref{eq:BCD}. Moreover, if
\begin{equation}\tag{\textup{CFL}}\label{eq:CFL}
\tau\leq \frac{s}{C_s}\veps^{2s},
\end{equation}
then the following properties hold:
\begin{enumerate}[\rm (a)]
\item\label{prop:propscheme-item1} \textup{($C_\textup{b}$-stability)} $\|U^j\|_{C_\textup{b}(\ren)}\leq \|u_0\|_{C_\textup{b}(\ren)}$.
\item\label{prop:propscheme-item3} \textup{(Comparison principle)}  Let $U^j$ and $V^j$ be sub- and supersolutions of \eqref{eq:EQD}--\eqref{eq:BCD} with respective initial data $u_0\in BUC(\ren)$ and $v_0\in BUC(\ren)$. If $u_0\leq v_0$ in $\ren$, then $U^j\leq V^j$ in $\ren$ for all $j\in \N$.
\item\label{prop:propscheme-item4} \textup{($C_\textup{b}$-contraction)}  Let $U^j$ and $V^j$ be solutions of \eqref{eq:EQD}--\eqref{eq:BCD} with respective initial data $u_0\in BUC(\ren)$ and $v_0\in BUC(\ren)$. Then
\begin{equation*}%\label{eq:Linftycontrac}
\|U^j-V^j\|_{C_\textup{b}(\ren)} \leq \|u_0-v_0\|_{C_\textup{b}(\ren)} \qquad \text{for all $j\in \N$}.
\end{equation*}
\item\label{prop:propscheme-item2} \textup{(Equicontinuity in space)} For all $y\in\ren$ and all $j\in\N$,
\item[]\qquad $\|U^j(\cdot+y)-U^j\|_{C_\textup{b}(\ren)}\leq  \omega_{u_0}(|y|).$
\item\label{prop:propscheme-contTime} \textup{(Equicontinuity in
  time)}  For all $j,k\in\N$ and all $0<\veps<1$, 
\item[]\qquad $\|U^{j+k}-U^j\|_{C_\textup{b}(\ren)}\leq
  \tilde\omega(|t_{j+k}-t_j|)$,\qquad where\qquad $\tilde\omega$ is defined in Theorem
\ref{thm:ExistenceAPrioriViscosity} (c).
\end{enumerate}
\end{proposition}

\begin{proof}
Since \eqref{eq:EQD}--\eqref{eq:BCD} is explicit and $
\ifleps:C_\textup{b}(\ren)\to C_\textup{b}(\ren)
$ is well-defined and bounded by Lemma \ref{lem:selfmap},  existence and
uniqueness follows by construction.

Let us then show the different a priori estimates:

\smallskip
\noindent\eqref{prop:propscheme-item3} Since $u_0\leq v_0$, we have $U^0\leq V^0$. Then, by induction assume that $U^j\leq V^j$. By \eqref{eq:EQD}, we get
\begin{equation*}%\label{eq:comp}
\begin{split}
U^{j+1}(x)-V^{j+1}(x)&\leq U^{j}(x)-V^{j}(x) + \tau \left(\ifleps[U^j](x)-\ifleps[V^j](x)\right)\\
&=\left(U^{j}(x)-V^{j}(x)\right)\left(1-\tau\frac{C_s}{s\veps^{2s}}\right) \\
&\quad+ \tau C_s \left(\sup_{|y|=1}\int_\veps^\infty U^j(x+\eta y) \frac{\dd t}{t^{1+2s}}  - \sup_{|y|=1}\int_\veps^\infty V^j(x+\eta y) \frac{\dd \eta }{\eta ^{1+2s}} \right)\\
&\quad+\tau C_s \left(\inf_{|y|=1}\int_\veps^\infty U^j(x+\eta y) \frac{\dd \eta }{\eta ^{1+2s}}  - \inf_{|y|=1}\int_\veps^\infty V^j(x+\eta y) \frac{\dd \eta }{\eta ^{1+2s}} \right)\\
&\leq 0,
\end{split}
\end{equation*}
where the last inequality follows from the induction hypothesis $U^j\leq V^j$ and \eqref{eq:CFL}.

\smallskip
\noindent\eqref{prop:propscheme-item1} Note that
\[
V^j:=\inf_{x\in \ren} \{u_0(x)\} \quad  \textup{and} \quad W^j:=\sup_{x\in \ren} \{u_0(x)\} \qquad \textup{for all $j\in \N$}
\]
are solutions of \eqref{eq:EQD}--\eqref{eq:BCD}. Since $\inf_{x\in \ren} \{u_0(x)\} \leq u_0 \leq \sup_{x\in \ren} \{u_0(x)\} $, we have  by \eqref{prop:propscheme-item3} that
\[
\inf_{x\in \ren} \{u_0(x)\}=V^j \leq U^j \leq W^j=\sup_{x\in \ren} \{u_0(x)\}.
\]

\smallskip
\noindent\eqref{prop:propscheme-item4} By the proof of \eqref{prop:propscheme-item3} and the fundamental inequalities $|\sup(\cdots)-\sup(\cdots)|\leq \sup(|\cdots-\cdots|)$ and $|\inf(\cdots)-\inf(\cdots)|\leq \sup(|\cdots-\cdots|)$, we can also get that
\begin{equation*}
\begin{split}
&|U^{j+1}(x)-V^{j+1}(x)|\\
&\leq \left|U^{j}(x)-V^{j}(x)\right|\left(1-\tau\frac{C_s}{s\veps^{2s}}\right)+ \tau C_s \left|\sup_{|y|=1}\int_\veps^\infty U^j(x+\eta y) \frac{\dd \eta }{\eta ^{1+2s}}  - \sup_{|y|=1}\int_\veps^\infty V^j(x+\eta y) \frac{\dd \eta }{\eta ^{1+2s}} \right|\\
&\quad+\tau C_s \left|\inf_{|y|=1}\int_\veps^\infty U^j(x+\eta y) \frac{\dd \eta }{\eta ^{1+2s}}  - \inf_{|y|=1}\int_\veps^\infty V^j(x+\eta y) \frac{\dd \eta }{\eta ^{1+2s}} \right|\\
&\leq \left|U^{j}(x)-V^{j}(x)\right|\left(1-\tau\frac{C_s}{s\veps^{2s}}\right) + 2\tau C_s \sup_{|y|=1}\int_\veps^\infty \left|U^j(x+\eta y)- V^j(x+\eta y) \right|\frac{\dd \eta }{\eta ^{1+2s}}\\
&\leq \left\|U^{j}-V^{j}\right\|_{C_\textup{b}(\ren)}\left(1-\tau\frac{C_s}{s\veps^{2s}}\right) + 2 \tau C_s \left\|U^{j}-V^{j}\right\|_{C_\textup{b}(\ren)} \int_\veps^\infty \frac{\dd \eta }{\eta ^{1+2s}}\\
&= \left\|U^{j}-V^{j}\right\|_{C_\textup{b}(\ren)}.
\end{split}
\end{equation*}
In this way we have proved that
\[
\left\|U^{j+1}-V^{j+1}\right\|_{C_\textup{b}(\ren)}\leq \left\|U^{j}-V^{j}\right\|_{C_\textup{b}(\ren)} \qquad \textup{for all $j\in \N$.}
\]
An iteration then concludes the proof.

\smallskip
\noindent\eqref{prop:propscheme-item2} This follows by using the translation invariant properties of \eqref{eq:EQD}--\eqref{eq:BCD} and part \eqref{prop:propscheme-item4}. More precisely, $W^j:=U^j(\cdot+y)$ is the unique solution of \eqref{eq:EQD}--\eqref{eq:BCD} with initial data $w_0:= u_0(\cdot+y)$ for all $y\in \ren$. Part \eqref{prop:propscheme-item4} then yields
\[
\|U^j(\cdot+y)-U^j\|_{C_\textup{b}(\ren)}=\|W^j-U^j\|_{C_\textup{b}(\ren)}\leq \|w_0-u_0\|_{C_{\textup{b}}(\ren)} =\|u_0(\cdot+y)-u_0\|_{C_\textup{b}(\ren)}.
\]

\noindent\eqref{prop:propscheme-contTime}
Consider a mollification of the initial data $u_{0,\delta}:=u_0\ast\rho_\delta$, and denote the corresponding solution by $U_\delta^j$. Choose $j=1$ in \eqref{eq:EQD}--\eqref{eq:BCD} to get
$$
\|U_\delta^1-U_\delta^0\|_{C_\textup{b}(\ren)}\leq \tau\|\ifleps[U_\delta^0]\|_{C_\textup{b}(\ren)}=\tau\|\ifleps[u_\delta^0]\|_{C_\textup{b}(\ren)}:= \tau K({u_{0,\delta}}).
$$
Now, define
$$
V_\delta^{j}:=U_\delta^{j+1} \qquad\textup{for all $j\in\N$}.
$$
Then $V_\delta^j$ is the unique solution of \eqref{eq:EQD}--\eqref{eq:BCD} with initial data $V_\delta^0=U_\delta^1$. By \eqref{prop:propscheme-item4},
\begin{equation}\label{eq:ArbitraryDistanceInTime}
\begin{split}
\|U_\delta^{j+1}-U_\delta^j\|_{C_\textup{b}(\ren)}= \|V_\delta^{j}-U_\delta^j\|_{C_\textup{b}(\ren)} \leq \|V_\delta^{0}-U_\delta^0\|_{C_\textup{b}(\ren)}=\|U_\delta^{1}-U_\delta^0\|_{C_\textup{b}(\ren)}\leq \tau K({u_{0,\delta}}).
\end{split}
\end{equation}
Repeated use of the triangle inequality then yields
$$
\|U_\delta^{j+k}-U_\delta^j\|_{C_\textup{b}(\ren)}\leq \sum_{i=0}^{k-1}\|U_\delta^{(j+i)+1}-U_\delta^{j+i}\|_{C_\textup{b}(\ren)}\leq k\tau K({u_{0,\delta}})=(t_{j+k}-t_j)K({u_{0,\delta}}).
$$
Then by \eqref{prop:propscheme-item4},
\begin{align*}
&\|U^{j+k}-U^j\|_{C_\textup{b}(\ren)}\leq \|U^{j+k}-U_\delta^{j+k}\|_{C_\textup{b}(\ren)}+\|U_\delta^{j+k}-U_\delta^j\|_{C_\textup{b}(\ren)}+\|U_\delta^{j}-U^j\|_{C_\textup{b}(\ren)}\\
&\leq
2\|u_0-u_{0,\delta}\|_{C_{\textup{b}}}+(t_{j+k}-t_j)K({u_{0,\delta}}) \leq 2\omega_{u_0}(\delta)+(t_{j+k}-t_j)K({u_{0,\delta}}),
\end{align*}
where we used that by properties of mollifiers, $\|u_0-u_{0,\delta}\|_{C_{\textup{b}}}\leq
\sup_{|y|\leq\delta}\|u_0(\cdot+y)-u_0\|_{C_\textup{b}(\ren)}\leq \omega_{u_0}(\delta)$. Hence
the result follows by the definition of $\tilde\omega$.
%
%% \begin{align*}
%% &\leq 2\sup_{|y|\leq\delta}\|u_0(\cdot+y)-u_0\|_{C_\textup{b}(\ren)}+\frac{C_s}{2(1-s)}\|u_0\|_{C_\textup{b}(\ren)}\|\nabla^2\rho\|_{L^1(\ren)}(t_{j+k}-t_j)\delta^{-2}\\
%% &\quad+\frac{2C_s}{s}\|u_0\|_{C_\textup{b}(\ren)}\|\rho\|_{L^1(\ren)}(t_{j+k}-t_j)
%% \end{align*}
%% The conclusion follows by choosing $\delta=(t_{j+k}-t_j)^{\frac{1}{3}}$.
\end{proof}

%%%%%%%%%%%%%%%%%%%%%%%%%%%%%%%%%%%%%%%%%%%%%%%%%%%%%%%%%%%%%
\subsection{Semi-discrete scheme defined on $\ren\times[0,\infty)$}

In order to get uniform convergence of our approximation scheme, we need to define it on $\ren\times[0,\infty)$. Let us therefore consider the solution of \eqref{eq:EQD}--\eqref{eq:BCD} $U_\veps:\ren\times \{\tau \N\,\cup\,0\}\to\R$ and the function $u_\veps:\ren\times[0,\infty)\to \R$ defined as:
\begin{equation*}
\begin{cases}
u_\veps(x,0):=U^{0}_\veps(x)=u_0(x),\\[0.1cm]
u_\veps(x,t):= \frac{t_{j+1}-t}{\tau}U^{j}_\veps(x)+  \frac{t-t_{j}}{\tau} U^{j+1}_\veps(x)\quad\text{if $t\in(t_j, t_{j+1}]$ with $j\in \N$.}
\end{cases}
\end{equation*}

\begin{corollary}[Well-posedness and properties]\label{cor:propscheme}
Under the assumptions of Proposition \ref{prop:propscheme}, there
exists a unique pointwise solution $u_\veps\in
BUC(\ren\times[0,\infty))$ of \eqref{eq:EQD}--\eqref{eq:BCD} with
  initial data $u_0\in BUC(\ren)$. The solution, moreover, enjoys
  $C_\textup{b}$-stability, comparison principle,
  $C_\textup{b}$-contraction, continuity in space, and continuity in
  time  in form of
$\|u_\veps(\cdot,t)-u_\veps(\cdot,\tilde{t})\|_{C_\textup{b}(\ren)}\leq
\tilde\omega(|t-\tilde t|)$ for all $t,\tilde{t}\in [0,\infty]$.
\end{corollary}

\begin{proof}  
We easily inherit all properties from $U_\veps$ to $u_\veps$, e.g.
\[
  \|u_\veps(\cdot,t)\|_{C_\textup{b}} 
  \leq
  \frac{(t_{j+1}-t)}{\tau} \left \|
  U^{j}_\veps\right\|_{C_\textup{b}}+ \frac{(t-t_{j})}{\tau} \left\|
  U^{j+1}_\veps\right\|_{C_\textup{b}} \\
 \leq \Big(\frac{(t_{j+1}-t)}{\tau}+  \frac{(t-t_{j})}{\tau}\Big) \left\| u_0\right\|_{C_\textup{b}}=  \left\| u_0\right\|_{C_\textup{b}}.
%\end{split}
\]
The other properties follows in a similar way, and we only explain
the most difficult one, the continuity in time. 
  Repeating the steps of the proof of Proposition
 \ref{prop:propscheme}\eqref{prop:propscheme-contTime}, for
 % , except that we now use the time interpolant. For
 $\tilde{t}\in(t_{j+k},t_{j+k+1}]$ and $t\in(t_{j},t_{j+1}]$,
\begin{equation*}
\begin{split}
  &|(u_\veps)_{\delta}(x,\tilde{t})-(u_\veps)_{\delta}(x,t)|\\
  &\leq |(u_\veps)_{\delta}(x,\tilde{t})-(u_\veps)_{\delta}(x,t_{j+k})|+| (u_\veps)_{\delta}(x,t_{j+1})-(u_\veps)_{\delta}(x,t)|  + \sum_{l={j+1}}^{j+k-1}|(u_\veps)_{\delta}(x,t_{l+1}) -  (u_\veps)_{\delta}(x,t_{l})|\\
&=|(u_\veps)_{\delta}(x,\tilde{t})-(U_\veps^{j+k})_{\delta}(x)|+|(U_\veps^{j+1})_{\delta}(x)-(u_\veps)_{\delta}(x,t) | + \sum_{l={j+1}}^{j+k-1}|(U_\veps^{l+1})_{\delta}(x) -  (U_\veps^{l})_{\delta}(x)|.
\end{split}
\end{equation*}
By the definition of linear interpolation
\begin{align*}
\|(u_\veps)_{\delta}(\cdot,\tilde{t})-(U_\veps^{j+k})_{\delta}\|_{C_\textup{b}(\ren)}&\leq\frac{(\tilde{t}-t_{j+k})}{\tau}\|(U_\veps^{j+k+1})_{\delta}-(U_\veps^{j+k})_{\delta}\|_{C_\textup{b}(\ren)},\\
\|(U_\veps^{j+1})_{\delta}-(u_\veps)_{\delta}(\cdot,t)\|_{C_\textup{b}(\ren)}&\leq \frac{(t_{j+1}-t)}{\tau}\|(U_\veps^{j+1})_{\delta}-(U_\veps^j)_{\delta}\|_{C_\textup{b}(\ren)},
\end{align*}
and then by repeated use of \eqref{eq:ArbitraryDistanceInTime},
\begin{equation*}
  \|(u_\veps)_{\delta}(\cdot,\tilde{t})-(u_\veps)_{\delta}(\cdot,t)\|_{C_\textup{b}(\ren)}
  %\\
%&=\|(u_\veps)_{\delta}(\cdot,\tilde{t})-(U_\veps^{j+k})_{\delta}\|_{C_\textup{b}(\ren)}+\|(U_\veps^{j+1})_{\delta}-(u_\veps)_{\delta}(\cdot,t) \|_{C_\textup{b}(\ren)}+ \sum_{l={j+1}}^{j+k-1}\|(U_\veps^{l+1})_{\delta} -  (U_\veps^{l})_{\delta}\|_{C_\textup{b}(\ren)}\\
\leq \Big((\tilde{t}-t_{j+k})+\sum_{l={j+1}}^{j+k-1}\tau+(t_{j+1}-t)\Big)K(u_{0,\delta})=(\tilde{t}-t)K(u_{0,\delta}).
\end{equation*}
We can then conclude the proof continuing as in the proof of Proposition
\ref{prop:propscheme}\eqref{prop:propscheme-contTime}.
%to \cm get $
%\|u_\veps(\cdot,t)-u_\veps(\cdot,\tilde{t})\|_{C_\textup{b}(\ren)}\leq
%\tilde\omega(|t-\tilde t|).
%2\sup_{|y|\leq |t-\tilde{t}|^{\frac{1}{3}}}\|u_0(\cdot+y)-u_0\|_{C_\textup{b}(\ren)}+C\big(|t-\tilde{t}|^{\frac{1}{3}}+|t-\tilde{t}|\big).
%$
%The proof is then complete.
\end{proof}

%%%%%%%%%%%%%%%%%%%%%%%%%%%%%%%%%%%%%%%%%%%%%%%%%%%%%%%%%%%%%
\subsection{Compactness in $UC_{\textup{\textnormal{loc}}}(\ren\times[0,\infty))$}

\begin{proposition}[Compactness]\label{prop:SemiDiscreteCompactness}
Under the assumptions of Proposition \ref{prop:propscheme}, there exists
a subsequence $\{u_{\veps_k}\}_{k\in\N}$ and a  $u\in
C_{\textup{b}}(\ren\times[0,\infty))$
  %ERJ DROP, UC_loc = C^0: B(\ren\times[0,\infty))\cap UC_\textup{loc}(\ren\times[0,\infty))$
such that
$$
u_{\veps_k}\to u\qquad\text{locally uniformly in $\ren\times[0,\infty)$ as $k\to\infty$.}
$$
% ERJ: Unif.conv. is stronger than pt.wise conv. So we should DROP: Moreover, the convergence is pointwise (up to possibly another subsequence).
\end{proposition}

\begin{proof}
  The sequence $\{u_{\veps}\}_\veps$ is equibounded  and
  equicontinuous by Corollary \ref{cor:propscheme} (see also Proposition
  \ref{prop:propscheme}). The result then follows from
  the Arzel\`a-Ascoli compactness theorem.
%  by the $C_\textup{b}$-stability in Corollary \ref{cor:propscheme}. To prove that it is also equicontinuous, we consider, for $x,y\in\ren$ and $t,\tilde{t}>0$ such that $|x-y|=|\xi|\leq \eta$ and $|t-\tilde{t}|\leq \kappa$,
%% \begin{equation*}
%% \begin{split}
%% |u_\veps(x,t)-u_\veps(y,\tilde{t})|&\leq |u_\veps(x,t)-u_\veps(x+(y-x),t)|+|u_\veps(y,t)-u_\veps(y,\tilde{t})|\\
%% &\leq \sup_{|\xi|\leq \eta}\|u_\veps(\cdot,t)-u_\veps(\cdot+\xi,t)\|_{C_\textup{b}(\ren)}+\|u_\veps(\cdot,t)-u_\veps(\cdot,\tilde{t})\|_{C_\textup{b}(\ren)}\\
%% &\leq \sup_{|\xi|\leq \eta}\|u_0(\cdot+\xi)-u_0\|_{C_{\textup{b}}(\ren)}+2\sup_{|\zeta|\leq \kappa^{\frac{1}{3}}}\|u_0(\cdot+\zeta)-u_0\|_{C_\textup{b}(\ren)}+C\big(\kappa^{\frac{1}{3}}+\kappa\big),
%% \end{split}
%% \end{equation*}
%% where we used Corollary \ref{cor:propscheme} again. A covering and diagonal argument then also gives pointwise convergence. Finally, by the $C_\textup{b}$-stability in Corollary \ref{cor:propscheme} and for $k$ big enough,
%% $$
%% |u(x,t)|\leq |u(x,t)-u_{\veps_k}(x,t)|+|u_{\veps_k}(x,t)|\leq 1+\|u_0\|_{C_\textup{b}(\ren)},
%% $$
%% which concludes the proof.
\end{proof}

 Taking pointwise limits in the a priori estimates of Corollary \ref{cor:propscheme} (see
also Propostion \ref{prop:propscheme}), the limit $u$ immediately inherits these
estimates.
%We immediately have:

\begin{corollary}[A priori estimates]\label{cor:PropertiesOfLimitFunction}
 Assume $u_0\in BUC(\ren)$. Then the limit $u$ from Propostion
\ref{prop:SemiDiscreteCompactness} enjoys the following properties:
\begin{enumerate}[\rm (a)]
\item \textup{($C_\textup{b}$-stability)} For all $t>0$, $\|u(\cdot,t)\|_{C_\textup{b}(\ren)}\leq \|u_0\|_{C_\textup{b}(\ren)}$.
\item \textup{(Uniform  continuity in space)} For all $y\in\ren$ and all $t>0$,
\item[] \qquad$\|u(\cdot+y,t)-u(\cdot,t)\|_{C_\textup{b}(\ren)} \leq \omega_{u_0}(|y|).$
\item \textup{(Uniform continuity in time)}  For all $t,\tilde{t}>0$,
\item[] \qquad $\|u(\cdot,t)-u(\cdot,\tilde{t})\|_{C_\textup{b}(\ren)}\leq \tilde
  \omega(t-\tilde t)$\qquad where $\tilde
    \omega$ is defined in Theorem \ref{thm:ExistenceAPrioriViscosity}.
\end{enumerate}
\end{corollary}

\section{Definitions, existence and properties of viscosity solutions}
\label{sec:existence}

%%%%%%%%%%%%%%%%%%%%%%%%%%%%%%%%%%%%%%%%%%%%%%%%%%%%%%%%%%%%%

In this section we define the concept of viscosity solution. Before
doing so we need to introduce two new operators that will be used when
testing at zero gradient points.

\begin{definition}
For $\phi\in C^{1,1}(x)\cap B(\ren)$,
\begin{align*}%\label{eq:def3a}
\iflp \phi(x)&:=C_s \sup_{|y|=1}\int_0^\infty\big(\phi\left(x+ \eta  y\right)+\phi\left(x- \eta  y\right)-2\phi(x)\big)\frac{\dd \eta }{\eta ^{1+2s}},\\
\iflm \phi(x)&:=C_s \inf_{|y|=1}\int_0^\infty\big(\phi\left(x+ \eta  y\right)+\phi\left(x- \eta  y\right)-2\phi(x)\big)\frac{\dd \eta }{\eta ^{1+2s}}.
\end{align*}
\end{definition}

We immediately have:

\begin{lemma}\label{lem:orderop}
For $\phi\in C^{1,1}(x)\cap B(\ren)$,
\[
\iflm \phi(x) \leq \ifl \phi(x) \leq \iflp \phi(x).
\]
\end{lemma}

\begin{proof}
Recall Lemma \ref{lem:EquivalentDefIFL}. The result is trivial unless $\nabla\phi(x)=0$. In that case,
\[\begin{split}
\ifl \phi(x)&=C_s\sup_{|y|=1} \int_{0}^\infty\big(\phi(x+\eta y)-\phi(x)\big) \frac{\dd \eta }{\eta ^{1+2s}} + C_s\inf_{|y|=1} \int_{0}^\infty\big(\phi(x-\eta y)-\phi(x)\big) \frac{\dd \eta }{\eta ^{1+2s}}\\
&=C_s\sup_{|y|=1} \int_{0}^\infty\big(\phi(x+\eta y)-\phi(x)\big) \frac{\dd \eta }{\eta ^{1+2s}} - C_s\sup_{|y|=1} \left(- \int_{0}^\infty\big(\phi(x-\eta y)-\phi(x)\big) \frac{\dd \eta }{\eta ^{1+2s}} \right)\\
&\leq C_s  \sup_{|y|=1}  \left(\Big( \int_{0}^\infty\big(\phi(x+\eta y)-\phi(x)\big) \frac{\dd \eta }{\eta ^{1+2s}}\Big) - \Big(- \int_{0}^\infty\big(\phi(x-\eta y)-\phi(x)\big) \frac{\dd \eta }{\eta ^{1+2s}} \Big)\right)\\
%&= C_s  \sup_{|y|=1}  \int_{0}^\infty\big(\phi(x+\eta y)+\phi(x-\eta y)-2\phi(x)\big) \frac{\dd \eta }{\eta ^{1+2s}} \\
&= \iflp \phi(x).
\end{split}
\]
The result $\iflm \phi(x) \leq \ifl \phi(x)$ follows in a similar way.
\end{proof}

We are now ready to define the concept of viscosity solution.

\begin{definition}[Viscosity solution]\label{def:viscsol}
\begin{enumerate}[\rm (a)]
\item A globally bounded upper semicontinuous function $u:\ren\times(0,\infty)\to \R$ is a \emph{viscosity subsolution} of \eqref{eq:EQ} if, for all $(x_0,t_0)\in \ren\times(0,\infty)$, all $\phi\in C^2(B_R(x_0,t_0))\cap BUC(\ren\times(0,\infty)\setminus B_R(x_0,t_0))$ for some $R>0$ and such that
\begin{enumerate}[\rm (i)]
\item $u(x_0,t_0)-\phi(x_0,t_0)=\sup_{(x,t)\in B_R(x_0,t_0)}\big(u(x,t)-\phi(x,t)\big)$,
\item $u(x_0,t_0)-\phi(x_0,t_0)>u(x,t)-\phi(x,t)$ for all $(x,t)\in B_R(x_0,t_0)\setminus (x_0,t_0)$,
\item $u(x_0,t_0)-\phi(x_0,t_0)\geq u(x,t)-\phi(x,t)$ for all $(x,t)\in \ren\times(0,\infty) \setminus B_R(x_0,t_0)$,
\end{enumerate}
then
\begin{empheq}[left=\empheqlbrace]{align}
\partial_t \phi(x_0,t_0) \leq & \  \ifl \phi(x_0,t_0), \hspace{-3cm}  &  \textup{if}  &\hspace{0.5cm}  \nabla \phi(x_0,t_0) \not=0  \nonumber\\
\partial_t \phi(x_0,t_0)\leq & \ \iflp \phi(x_0,t_0), \hspace{-3cm}  &   \textup{if}&  \hspace{0.5cm}  \nabla \phi(x_0,t_0) =0 . \label{eq:viscsub2}
\end{empheq}
\item A globally bounded lower semicontinuous function $u:\ren\times(0,\infty)\to \R$ is a \emph{viscosity supersolution} of \eqref{eq:EQ} if, for all $(x_0,t_0)\in \ren\times(0,\infty)$, all $\phi\in C^2(B_R(x_0,t_0))\cap BUC(\ren\times(0,\infty)\setminus B_R(x_0,t_0))$ for some $R>0$ and such that
\begin{enumerate}[\rm (i)]
\item $u(x_0,t_0)-\phi(x_0,t_0)=\inf_{(x,t)\in B_R(x_0,t_0)}(u(x,t)-\phi(x,t))$,
\item $u(x_0,t_0)-\phi(x_0,t_0)<u(x,t)-\phi(x,t)$ for all $(x,t)\in B_R(x_0,t_0)\setminus (x_0,t_0)$,
\item $u(x_0,t_0)-\phi(x_0,t_0)\leq u(x,t)-\phi(x,t)$ for all $(x,t)\in \ren\times(0,\infty) \setminus B_R(x_0,t_0)$,
\end{enumerate}
then
\begin{empheq}[left=\empheqlbrace]{align}
\partial_t \phi(x_0,t_0) \geq & \  \ifl \phi(x_0,t_0), \hspace{-3cm}  &  \textup{if}  &\hspace{0.5cm}  \nabla \phi(x_0,t_0) \not=0  \nonumber\\
\partial_t \phi(x_0,t_0)\geq & \ \iflm \phi(x_0,t_0), \hspace{-3cm}  &   \textup{if}&  \hspace{0.5cm}  \nabla \phi(x_0,t_0) =0 . \label{eq:viscsuper2}
\end{empheq}
\item A  function  $u\in C_{\textup{b}}(\ren\times[0,\infty))$ is a
  viscosity solution of \eqref{eq:EQ} if it is both a viscosity
  subsolution and a viscosity supersolution. 
\item The viscosity solution $u$ takes the initial data in a pointwise way: $u(x,0)=u_0(x)$ for all $x\in\ren$. 
\end{enumerate}
\end{definition}

\begin{remark}\label{rem-visc} %We have several remarks to do:
\begin{enumerate}[\rm (a)]
\item  In points where the gradient of $u$ is zero, we only require 
\(
\partial_tu\in[ \iflm u,  \iflp u].
\)

\item In the  local elliptic and homogeneous  case
  \cite{JuLiMa01}, comparison follows   without any 
  %  even if one do not put any
  condition at points where $\nabla
  \phi=0$. In more general cases conditions are needed. We impose
  conditions \eqref{eq:viscsub2} and 
  \eqref{eq:viscsuper2} which are generalisations of the conditions
  introduced in the local parabolic case \cite{AkJuKa09}. 
  
  It is easy to show that comparison and uniqueness cannot hold
  %construct a counterexample \cm  if no
  without such conditions: E.g. $u(x,t)=1$ and $v(x,t)=2\sin(t)$ would
  then both be 
  viscosity solutions of \eqref{eq:EQ} since $\nabla u=\nabla v=0$ at
  every point.
  %% for $(x,t)\in 
  %% \ren\times[0,\infty)$. Since $\nabla u=\nabla v=0$ at every point,
  %%   %for     all $(x,t)\in \ren\times[0,\infty)$,
  %%   these functions  would  be
  %%     viscosity solutions of  \eqref{eq:EQ} if  we do not need to test
  %%     at zero gradient points, since we would not need need to test in
  %%     any point.
 However comparison does not hold since %     Moreover, the initial data satisfy 
\(
u(x,0)=1 \geq 0 =v(x,0)
\)
%However,
while
\(
u(x,\pi/2)=1\leq 2 = v(x,\pi/2).
\) 

Let us check that $v$ is not longer a viscosity solution when we
impose \eqref{eq:viscsub2}. Let $K>0$ and $0\leq \psi\in C^2_{\textup{b}}(\R)$ be radial with
$\psi(r)=r^2$ for $|r|<1$ and $\psi(r)=0$ for $|r|>2$. We define
%% \[
%% \phi(x,t)=v(\cm 0,2\pi\nc) + \frac{(1-s)}{C_s} |x|^2(1-|x|^2) \indik_{|\cdot|\leq1}(x) + -(t-2\pi)^2(1-(t-\pi)^2) \indik_{|\cdot-t_0|\leq1}(t).
%% \]
\[
\phi(x,t)=%\frac{(2-2s)}{3C_s\|\psi''\|_{C_b}}
v(t)+K\psi(|x|)+ \psi(t-t_0)\qquad\text{for some $t_0\in(0,\infty)$.}
\]
It is then immediate that $v-\phi$ has a strict local max at $(0,t_0)$ and $\nabla \phi=\nabla v=0$. Now let $t_0=2\pi$, then $\partial_{t}\phi(0,2\pi)= \partial_{t}v(0,2\pi)=2\cos(2\pi)=2$, and by radial symmetry and followed by compact support of $\phi$ leads to 
\[
\begin{split}
\iflp \phi(0,2\pi)&= KC_s\int_{0}^\infty \big( \phi(\eta e_1, 2\pi)+\phi(-\eta e_1, 2\pi)-2\phi(0,2\pi)\big)\frac{\dd \eta}{\eta^{1+2s}}\\
&\leq KC_s \int_{0}^2 \|\psi''\|_{C_{\textup{b}}}
\eta^2 \frac{\dd \eta}{\eta^{1+2s}}%=% (2-2s)  \int_{0}^1 \eta^{1-2s}
% \dd \eta
=KC_s\|\psi''\|_{C_{\textup{b}}}\frac{2^{2-2s}}{(2-2s)}\leq 1
\end{split}
\]
if $K$ is small enough.
This contradicts \eqref{eq:viscsub2} since
\[
\partial_{t}\phi(0,2\pi)=2 \geq1\geq  \iflp \phi(0,2\pi).
\]
We conclude that $v(x,t)=2\sin(t)$ is not a viscosity subsolution in the sense of Definition \ref{def:viscsol}.
\item\label{rem-visc-item4} Since $\partial_t$, $\ifl$, $\iflp$ and
  $\iflm$ are invariant under 
   translations of $\phi$ by constants,  without loss of generality, we can replace the conditions on the
  test function in Definition \ref{def:viscsol} by
\begin{enumerate}[(i')]
\item\label{test-mod-item1} $\phi(x_0,t_0)=u(x_0,t_0)$,
\item\label{test-mod-item2} $\phi>u$ (resp. $\phi<u$) in $ B_R(x_0,t_0)\setminus (x_0,t_0)$,
\item\label{test-mod-item3} $\phi\geq u$ (resp. $\phi\leq u$) in $ \ren\times(0,\infty) \setminus B_R(x_0,t_0)$.
\end{enumerate}

\item\label{rem-visc-item5} We can also assume that the max is
  globally strict by adding a small in $C_{\textup{b}}^2$ perturbation to $\phi$
  supported in $B_R^c$, e.g. replacing $\phi$ by $\phi+\delta \psi$
  where $\psi\in 
  C_{\textup{b}}^2$ is such that $0\leq \psi\leq 1$, $\psi=0$ in  $B_R$ and $\psi>0$ in
  $B_R^c$. This new test function also satisfies
  \eqref{rem-visc-item4}, but with a strict inequality in part
  (iii'). Moreover, at the max point local derivatives up to order 2
  coinside with those of $\phi$, while nonlocal derivatives differ by
  an $O(\delta)$ term since $|\ifl\psi|+|\iflp \psi|+|\iflm \psi|\leq
  C \|\psi\|_{C^2_{\textup{b}}}$. Before concluding the proof we then need to
  send $\delta\to0$. Since this is never a problem, we will omit this
  modification in some proofs and simply assume globally strict max in
  the definition of viscosity solutions.
\end{enumerate}
\end{remark}

\subsection{Existence  and properties  of
  solutions}\label{subsec:existence}

%% \cm We start with an interpolations result.

%% \begin{lemma}\label{lem:interp}
%%   If $\phi\in C_b^2(\R^N)$ and $2s\in(1,2)$, then there is $c(s,\rho)$
%%   such that for all $\veps>0$,
%%   $$\|\ifleps \phi\|_{C_b}\leq c(s)\|D\phi\|_{C_b}^{2-2s}\|D^2\phi\|_{C_b}^{2s-1}. $$
%%   \end{lemma}
%% \begin{proof}
%% In the expression for $\ifleps \phi$, we split the domain of
%% integration in two, $(\veps,r)$ and $(r,\infty)$. Then we Taylor expand to
%% second and first order respectively. The result 
%% is the following:
%% $$\|\ifleps \phi\|_{C_b}\leq \|D^2\phi\|\int_\veps^r \eta^2
%% \frac{d\eta}{\eta^{1+2s}}+2\|D\phi\|\int_r^\infty \eta
%% \frac{d\eta}{\eta^{1+2s}}\leq
%% c(s)\Big(r^{2-2s}\|D^2\phi\|+r^{1-2s}\|D\phi\|\Big).$$
%% Minimizing with respect to $r$ then gives the result.
%%   \end{proof}
%% \nc

We are now in a position to prove %the existence part of
Theorem \ref{thm:ExistenceAPrioriViscosity}.  

\begin{proof}[%The existence part
    Proof of Theorem \ref{thm:ExistenceAPrioriViscosity}]
By Proposition \ref{prop:SemiDiscreteCompactness} there is a 
subsequence $u_\veps$ such that
\[
u_\veps \to u \qquad \textup{locally uniformly in $\ren\times [0,\infty)$ as $\veps\to0^+$,}
\]
 and by Corollary \ref{cor:PropertiesOfLimitFunction},  $u\in
BUC(\ren\times[0,\infty))$ and satisfies the a priori estimates in Theorem
  \ref{thm:ExistenceAPrioriViscosity}. Let us prove that $\tilde
  \omega$ has the property claimed in part (c). To do this we use two
  basic facts about mollifiers: Since $u_0\in BUC(\ren)$, it
  follows that (differentiate $\rho$, use Young for convolutions) 
$$
\|D^ku_{0,\delta}\|_{C_{\textup{b}}}\leq
  \|D^k\rho\|_{L^1}\|u_0\|_{C_{\textup{b}}}\delta^{-k}\qquad\text{for $k\in\N$.}
  $$
  Then by Lemma \ref{lem:unifBoundLveps}, we have the following bound:
 $$
 \|\ifleps [u_{0,\delta}]\|_{C_{\textup{b}}}\leq c(s)
  \|\nabla\rho\|_{L^1}^{2-2s}\|D^2\rho\|_{L^1}^{2s-1}\|u_0\|_{C_{\textup{b}}}\delta^{-2s}. 
  $$
 Since $\delta^{-2s}\leq 1+\delta^{-2}$, the estimate on
 $\tilde\omega$ follows after taking $\delta=r^{1/3}$. %(cf. Corollary
  %\ref{cor:PropertiesOfLimitFunction}).
 
  It remains to check that $u$ is  a viscosity solution according
  to Definition \ref{def:viscsol}. By Remark \ref{rem-visc}\eqref{rem-visc-item5}, consider $(x_0,t_0)\in \ren\times(0,\infty)$ and $\phi\in C^2(B_R(x_0,t_0))\cap BUC(\ren\times(0,\infty)\setminus B_R(x_0,t_0))$ such that
\begin{enumerate}[\rm (i)]
\item $u(x_0,t_0)-\phi(x_0,t_0)=\sup_{(x,t)\in B_R(x_0,t_0)}(u(x,t)-\phi(x,t))$,
\item $u(x_0,t_0)-\phi(x_0,t_0)>u(x,t)-\phi(x,t)$ for all $(x,t)\in \ren\times(0,\infty)\setminus (x_0,t_0)$.
\end{enumerate}
Local uniform convergence ensures that there exists a sequence $\{(x_\veps,t_\veps)\}_{\veps>0}$ such that
\begin{enumerate}[\rm (i)]
\item $u_\veps(x_\veps,t_\veps)-\phi(x_\veps,t_\veps)=\sup_{(x,t)\in B_R(x_\veps,t_\veps)}(u(x,t)-\phi(x,t)):=M_\veps$,
\item $u_\veps(x_\veps,t_\veps)-\phi(x_\veps,t_\veps)>u_\veps(x,t)-\phi(x,t)$ for all $(x,t)\in \ren\times(0,\infty)\setminus (x_\veps,t_\veps)$,
\end{enumerate}
and
\[
(x_\veps,t_\veps)\to (x_0,t_0) \qquad \textup{as $\veps\to0$.}
\]

Recall that Corollary \ref{cor:propscheme} ensures that $u_\veps$ solves the semidiscrete scheme. For simplicity, we use the notation in Remark \ref{remark:schemeform}. Let $t_j$ be such that $t_\veps\in(t_j,t_{j+1}]$.  It is standard to check that 
\[
\frac{u_\veps(x_\veps,t_\veps)-u_\veps(x_\veps,t_j)}{t_\veps-t_j} = L(\veps, u_\veps, u_\veps(x_\veps,t_j)).
\]
or equivalently
\[
\frac{(u_\veps(x_\veps,t_\veps)-M_\veps)-(u_\veps(x_\veps,t_j)- M_\veps)}{t_\veps-t_j} = L(\veps, u_\veps-M_\veps, u_\veps(x_\veps,t_j)-M_\veps)
\]
By defining $\tilde{u}_\veps:=u_\veps-M_\veps$, we have that $\tilde{u}_\veps(x_\veps,t_\veps)=\phi(x_\veps,t_\veps)$ and $\phi> \tilde{u}_\veps$ in $\ren\times [0,\infty)\setminus (x_\veps,t_\veps)$. Let us then rewrite the scheme to obtain
\[
\begin{split}
\phi(x_\veps,t_\veps)&=\tilde{u}_\veps(x_\veps,t_\veps)\\
&= \tilde{u}_\veps(x_\veps,t_j)\left(1- (t_\veps-t_j) \frac{C_s}{s \veps^{2s}}\right) \\
&\quad+ (t_\veps-t_j) C_s \left(\sup_{|y|=1}\int_\veps^\infty \tilde{u}_\veps(x_\veps+\eta y,t_j)  \frac{\dd \eta}{\eta^{1+2s}}  + \inf_{|y|=1}\int_\veps^\infty \tilde{u}_\veps(x_\veps+\eta y,t_j)  \frac{\dd \eta}{\eta^{1+2s}}\right)\\
&<  \phi(x_\veps,t_j)\left(1- (t_\veps-t_j) \frac{C_s}{s \veps^{2s}}\right) \\
&\quad+ (t_\veps-t_j) C_s \left(\sup_{|y|=1}\int_\veps^\infty\phi(x_\veps+\eta y,t_ju)  \frac{\dd \eta}{\eta^{1+2s}}  + \inf_{|y|=1}\int_\veps^\infty \phi(x_\veps+\eta y,t_j)  \frac{\dd \eta}{\eta^{1+2s}}\right)
\end{split}
\]
that is,
\begin{align}\label{appreq}
\frac{\phi(x_\veps,t_\veps)-\phi(x_\veps,t_j)}{t_\veps-t_j}<  L(\veps, \phi, \phi(x_\veps,t_j)).
\end{align}

 Assume $\nabla \phi(x_0,t_0) >0$ (the $\nabla \phi(x_0,t_0)
<0$ case is similar). Then for $\veps_0>0$ small enough, $\nabla
\phi(x_\veps,t_j) >0$ for $\veps\leq \veps_0$, and we use
\eqref{appreq} og Lemma \ref{lem:Consistency} to find that
\begin{equation}\label{eq:almostvisc}
\partial_t\phi(x_\veps,t_\veps)\leq \ifl \phi(x_\veps,t_j ) + o_\veps(1)+ O(\tau),
\end{equation}
where $o_\veps$ depends on $\sup_{\veps\leq \veps_0}|\nabla
\phi(x_\veps,t_j)|^{-1}$ which is uniformly bounded by the above
discusion. Since $\phi$ is smooth, for every  $\eta \in [0,\infty)$,
\[
\phi\left(x_\veps \pm \frac{\nabla \phi(x_\veps)}{|\nabla \phi(x_\veps)|} \eta ,t_j\right) \stackrel{\veps\to0^+}{\longrightarrow} \phi\left(x_0 \pm \frac{\nabla \phi(x_0)}{|\nabla \phi(x_0)|} \eta ,t_0\right).\]
The dominated convergence theorem then ensures that
\[
\ifl \phi(x_\veps,t_j)\to \ifl \phi(x_0,t_0) \quad \textup{as $\veps \to 0^+$.}
\]
We thus pass to the limit in $\eqref{eq:almostvisc}$  and get the
correct viscosity subsolution inequality.

When $\nabla \phi(x_0,t_0)=0$ we have  (see proof of Lemma \ref{lem:orderop})
\[
\begin{split}
L(\veps, \phi, \phi(x_\veps,t_j)) &\leq \sup_{|y|=1}\int_{\veps}^\infty \big(\phi(x_\veps+\eta y, t_j) + \phi(x_\veps-\eta y,t_j) -2\phi(x_\veps,t_j)\big) \frac{\dd \eta}{\eta^{1+2s}}\\
&= \iflp \phi(x_\veps, t_j)+ o_\veps(1),
\end{split}
\]
 and it only remains to check 
\(
 \iflp \phi(x_\veps,t_j) \stackrel{\veps\to0^+}{\longrightarrow}  \iflp \phi(x_0,t_0).
\)
To do that, note that
\[
\begin{split}
&C_s^{-1}|\iflp \phi(x_\veps,t_j)-\iflp \phi(x_0,t_0)|\\
&\leq \sup_{|y|=1} \bigg|\int_{0}^\infty\bigg(  \Big(\phi(x_\veps+\eta y,t_j) -  \phi(x_0+\eta y,t_0) \Big)
+ \Big(\phi(x_\veps-\eta y,t_j) - \phi(x_0-\eta y,t_0)  \Big)\\
&\quad- 2\Big(\phi(x_\veps,t_j) -  \phi(x_0,t_0)\Big) \bigg)\frac{\dd \eta}{\eta^{1+2s}}\bigg|\\
 &\leq\sup_{|y|=1}  \left|\int_0^\frac{R}{4}(\ldots) \frac{\dd \eta}{\eta^{1+2s}}\right|+  \sup_{|y|=1}\left|\int_\frac{R}{4}^\infty (\ldots)\frac{\dd \eta}{\eta^{1+2s}}\right|=:I_\veps^1+I_\veps^2
\end{split}
\]
Since $\phi\in BUC(\ren\times(0,\infty))$,
\[
I_\veps^2\leq 4\, \omega_{\phi} ((x_\veps-x_0, t_j-t_0)) \int_{\frac{R}{4}}^\infty \frac{\dd \eta}{\eta^{1+2s}} \to 0 \qquad \text{as $\veps\to0^+$.}
\]
 Then note that $I^1_\veps\leq \int_0^{\frac R4} F_\veps(\eta) \dd
\eta$ where $F_\veps(\eta)=\sup_{|y|=1}\dfrac{|(\dots)|}{\eta^{1+2s}}$.
By a second order Taylor expansion and continuity of all involved functions,
$$|F_\veps(\eta)|\leq \|D^2 \phi\|_{C_\textup{b}(B_{R/2}(x_0,t_0))}
\eta^{1-2s}\quad \text{for $\veps$ small,\quad and}\quad F_\veps(\eta)\to 0  
\quad\text{pointwise as}\quad \veps\to 0.$$
By the dominated convergence theorem it follows that $I^1_\veps\to0$.
Finally, the initial condition trivially holds since
%since locally uniform convergence implies pointwise convergence, we also get
\(
u(x,0)=\lim_{\veps\to0^+} u_\veps(x,0)=u_0(x).%\qedhere
\)
\end{proof}

It remains to proof Lemma \ref{lem:Holder} and then also Corollary \ref{thm:ExistenceAPrioriViscosity_cor} is proved.

\begin{proof}[Proof of Lemma \ref{lem:Holder}]
 Since $u_0\in C^{0,\beta}(\ren)$, $\omega_{u_0}(\delta)\leq
 |u_0|_{C^{0,\beta}}\delta^\beta$, and basic facts about mollifiers
 yields
$$ \|D^ku_{0,\delta}\|_{C_{\textup{b}}}\leq
  c(\rho)|u_0|_{C^{0,\beta}}\delta^{-k+\beta}\quad\text{for}\quad
  k\in\N,$$
see e.g. \cite{Kry96} and \cite[Appendix A]{DTLi22b}. Then, by Lemma \ref{lem:unifBoundLveps}, we have
 $\|\ifleps [u_{0,\delta}]\|_{C_{\textup{b}}}\leq c(s,\rho)\delta^{-2s+\beta}$.
\end{proof}

\section{Review of basic results on the fractional heat equation}\label{sec:heat}

Here we collect some well-known results on the fractional heat equation that we will need, see e.g. \cite{BlumGet1960, DrGaVo03, BaPeSoVa14, BoSiVa17}. The one-dimensional problem we consider is
\begin{empheq}[left=\empheqlbrace]{align}
\partial_t v(x,t)+  \flx v(x,t)&=0, \hspace{-3cm} &x \in \R& ,\, t > 0 ,\label{eq:FHEQ}\\
v (x,0) &= v_0(x),  \hspace{-3cm} & x \in \R& . \label{eq:FHBC}
\end{empheq}
The fundamental solution of \eqref{eq:FHEQ} is given by
\[
P_{s}(x,t)= \mathcal{F}^{-1}(\textup{e}^{-|\xi|^{2s}t})(x)
\]
where $\mathcal{F}$ denotes the Fourier transform and $\mathcal{F}^{-1}$ its inverse. Since the Fourier symbol $\textup{e}^{-|\xi|^{2s}t}$ is a tempered distribution, it follows that
\[
P_{s}\in C^\infty(\R\times(0,\infty)).
\]
Moreover, it is well-known that
\beq\label{fht.profile}
P_{s}(x,t)=t^{-\frac{1}{2s}}F(|x|t^{-\frac{1}{2s}}),
\eeq
with a profile $F(r)$ that is a smooth and strictly decreasing function of $r>0$. We can also deduce that, for all $\tau>0$, there exist constants  $c_1,c_2>0$ depending only on $s$, such that
\begin{equation}\label{eq:FHEQBounds}
c_1\frac{t}{(t^{\frac{1}{s}}+|x|^2)^{\frac{1+2s}{2}}}\leq P_{s}(x,t)\leq c_2\frac{t}{(t^{\frac{1}{s}}+|x|^2)^{\frac{1+2s}{2}}}\qquad\text{for all $(x,t)\in \R\times[\tau,\infty)$}.
\end{equation}
\normalcolor

Once the basic properties of the fundamental solution are established, we also recall that given any $0\leq v_0\in L^1(\R)$ (actually a bigger class can be considered),  the unique (very weak) solution of \eqref{eq:FHEQ}--\eqref{eq:FHBC} is given by convolution as
\begin{equation}\label{eq:FHEQConvolution}
v(x,t)=(P_{s}(\cdot,t)\ast v_0)(x)= \int_{-\infty}^\infty P_{s}(x-y,t) v_0(y)\dd y.
\end{equation}
Actually, since it is obtained by convolution with a $C^\infty$ kernel, the solution with nonnegative $L^1$-initial data will be $C^\infty$ smooth in $\R\times(0,\infty)$. We will also need:

\begin{lemma}[Classical solutions]\label{lem:FHEQradialsol}
Let $v_0\in C^\infty_\textup{b}(\R)$ be radially symmetric and
radially nonincreasing. Then there exists a unique solution $v \in
C^\infty_\textup{b}(\R\times[0,\infty))$ of
  \eqref{eq:FHEQ}--\eqref{eq:FHBC}. % with $v_0$ as initial data.
  Moreover, $v$ is radial, radially nonincreasing, and given by \eqref{eq:FHEQConvolution}.
\end{lemma}

Let us also recall Theorem 8.1 in \cite{BoSiVa17}.

\begin{lemma}[Global Harnack principle]\label{lem:finetailHE}
Let $v$ be the very weak solution of \eqref{eq:FHEQ}--\eqref{eq:FHBC} with initial data $v_0\in L^1(\R)$ such that $v_0\not\equiv0$ and
\[
0\leq v_0(x)\leq (1+|x|^2)^{-\frac{1+2s}{2}} \qquad \text{for all $|x|\geq R\geq 1$.}
\]
Then, for all $\tau>0$, there exist constants $k_1,k_2>0$ depending only on $s$ and $R$, such that
\[
k_1 \|v_0\|_{L^1(\R)} P_s(x,t) \leq v(x,t) \leq k_2 \|v_0\|_{L^1(\R)} P_s(x,t)\qquad\text{for all $(x,t)\in\R\times[\tau,\infty)$.}
\]
 Moreover, by \eqref{eq:FHEQBounds},  for all $\tau>0$, there exist constants $C_1,C_2>0$ depending only on $s$, $R$, and $\|v_0\|_{L^1(\R)}$, such that
\[
C_1\frac{t}{(t^{\frac{1}{s}}+|x|^2)^{\frac{1+2s}{2}}} \leq v(x,t) \leq
C_2\frac{t}{(t^{\frac{1}{s}}+|x|^2)^{\frac{1+2s}{2}}}\qquad\text{for all $(x,t)\in\R\times[\tau,\infty)$.}
\]
\end{lemma}

\section{Smooth solutions and the 1d fractional heat equation }\label{sec:radialsol}
  %The equation for radially symmetric solutions}

This section investigates different smooth solutions of \eqref{eq:EQ}--\eqref{eq:BC}.

\subsection{Radial solutions}

We will now focus on obtaining Theorem \ref{thm:ExistenceClassicalSolution}. To do so, we will demonstrate that for radially symmetric and radially nonincreasing functions, the operator $\ifl$ reduces to the classical fractional Laplacian.

\begin{proposition}\label{prop:radialOp}
Assume that $\phi\in C^{1,1}(x)\cap B(\ren)$ is radial  and radially nonincreasing, i.e.,
\[
\phi(x)=\tilde\Phi(|x|) \qquad \text{for all $x\in \ren$,}
\]
where $\tilde\Phi:[0,\infty)\to \R$ is nonincreasing.  Then
\[
\ifl \phi(x)= -\fl \Phi(|x|),
\]
where $\Phi$ is the even extension of $\tilde\Phi$ to $\R$: 
$\Phi(r)=\tilde\Phi(r)$  and $\Phi(-r)=\Phi(r)$ for $r\in[0,\infty)$.
\end{proposition}

% for all $r\geq0$ such that  and $\Phi$ is , i.e., $\Phi(r_1)\geq \Phi(r_2)$ if $0\leq r_1\leq r_2$.

\begin{remark}
\begin{enumerate}[{\rm (a)}]
\item When $\phi$ is radial and $\nabla \phi(x)\neq0$, a similar observation has been done in the proof of Lemma 3.1 in \cite{BjCaFi12}.
\item ``Radially nonincreasing'' is needed only when $\nabla \phi(x)=0$, but it cannot be removed in general; see the Section \ref{sec:radialOpCounter} below.
\end{enumerate}
\end{remark}

\begin{proof}[Proof of Proposition \ref{prop:radialOp}]
Assume $\nabla\phi(x)\not=0$, and let $r=|x|$. We have that
\[
\nabla \phi(x)=\frac{\Phi'(r)}{r}x, \qquad |\nabla
\phi(x)|=|\Phi'(r)|, \qquad
%\text{and} \qquad
\frac{\nabla \phi(x)}{|\nabla \phi(x)|}=\pm \frac{x}{|x|},
\]
Lemma \ref{lem:EquivalentDefIFL} then yields
\begin{equation}\label{eq:1Difltofl}
\begin{split}
  \ifl \phi(x)
  %&=C_s\int_0^\infty\left(\phi\left(x+ \eta  \frac{x}{|x|} \right)+\phi\left(x- \eta  \frac{x}{|x|}\right)-2\phi(x)\right)\frac{\dd \eta }{\eta ^{1+2s}}\\
&=C_s\int_0^\infty\left(\phi\left(x(1+\frac{\eta }{|x|})\right)+\phi\left(x(1-\frac{\eta }{|x|})\right)-2\phi(x)\right)\frac{\dd \eta }{\eta ^{1+2s}}\\
&=C_s\int_0^\infty\left(\Phi\left(|x|(1+\frac{\eta }{|x|})\right)+\Phi\left(|x|(1-\frac{\eta }{|x|})\right)-2\Phi(|x|)\right)\frac{\dd \eta }{\eta ^{1+2s}}\\
&=C_s\int_0^\infty\left(\Phi\left(r+\eta \right)+\Phi\left(r-\eta \right)-2\Phi(r)\right)\frac{\dd \eta }{\eta ^{1+2s}}\\
&= -\fl \Phi(r).
\end{split}
\end{equation}

Assume $\nabla\phi(x)=0$. Note that
\begin{align*}
\dist(0, \partial B_\eta (x)) = \dist \left(0, x- \eta \frac{x}{|x|}\right),
\end{align*}
which implies that
\[
\sup_{z\in \partial B_\eta (x)}\{\phi(z)\}=\phi\left(x- \eta \frac{x}{|x|}\right),
\]
since $\phi=\Phi(|\cdot|)$ is radially nonincreasing. Then,
\[
\begin{split}
\int_{0}^\infty\left(\phi \left(x- \eta  \frac{x}{|x|}\right)-\phi(x) \right) \frac{\dd \eta }{\eta ^{1+2s}} &\leq \sup_{|y|=1}\int_{0}^\infty \left(\phi \left(x+ \eta  y\right)-\phi(x) \right) \frac{\dd \eta }{\eta ^{1+2s}}\\
&\leq \int_{0}^\infty \left(\sup_{z\in \partial B_\eta (x)}\{\phi(z)\}-\phi(x) \right) \frac{\dd \eta }{\eta ^{1+2s}}\\
&=\int_{0}^\infty \left(\phi\left(x- \eta \frac{x}{|x|}\right)-\phi(x) \right) \frac{\dd \eta }{\eta ^{1+2s}},
\end{split}
\]
so that
\[
\sup_{|y|=1}\int_{0}^\infty \left(\phi \left(x+ \eta  y\right)-\phi(x) \right) \frac{\dd \eta }{\eta ^{1+2s}}=\int_{0}^\infty \left(\phi\left(x- \eta \frac{x}{|x|}\right)-\phi(x) \right) \frac{\dd \eta }{\eta ^{1+2s}}.
\]
In the same way,
\[
\inf_{|y|=1}\int_{0}^\infty \left(\phi \left(x+ \eta  y\right)-\phi(x) \right) \frac{\dd \eta }{\eta ^{1+2s}}=\int_{0}^\infty \left(\phi\left(x+ \eta \frac{x}{|x|}\right)-\phi(x) \right) \frac{\dd \eta }{\eta ^{1+2s}}.
\]
Finally, Lemma \ref{lem:EquivalentDefIFL} and the argument in \eqref{eq:1Difltofl} gives the result.
\end{proof}

\begin{proof}[Proof of Theorem \ref{thm:ExistenceClassicalSolution}]
 Define $v_0(r):=U_0(|x|)$ for $x\in \ren$ and $r=|x|$, then $v_0\in
C_\textup{b}^\infty(\R)$ is radial and radially nonincreasing. Let $v$
be the corresponding solution of \eqref{eq:FHEQ}--\eqref{eq:FHBC}. % with $v_0$ as initial data. Then, b
By Lemma \ref{lem:FHEQradialsol}, $v$ is radial, radially
nonincreasing, and $C_\textup{b}^\infty$ smooth. Then by Proposition
\ref{prop:radialOp}, $u(x,t):=v(|x|,t)$ is a classical solution of
\eqref{eq:EQ}--\eqref{eq:BC}.
\end{proof}

\subsection{Counterexample for functions not being radially nonincreasing}\label{sec:radialOpCounter}

We show now an example of a function $\phi$ not satisfying the radially nonincreasing assumption in the zero gradient case, and such that the operators $\ifl$ and $-\fl$ do not coincide. For $R>0$, consider the radial function $\phi:\R^2\to \R$ given by $\phi(x) = \indik_{B_{R+1}\setminus B_{R-1}}(x)$, see  Figure \ref{fig:radialcounter}.

\begin{figure}[h!]
\center
\includegraphics[width=0.4\textwidth]{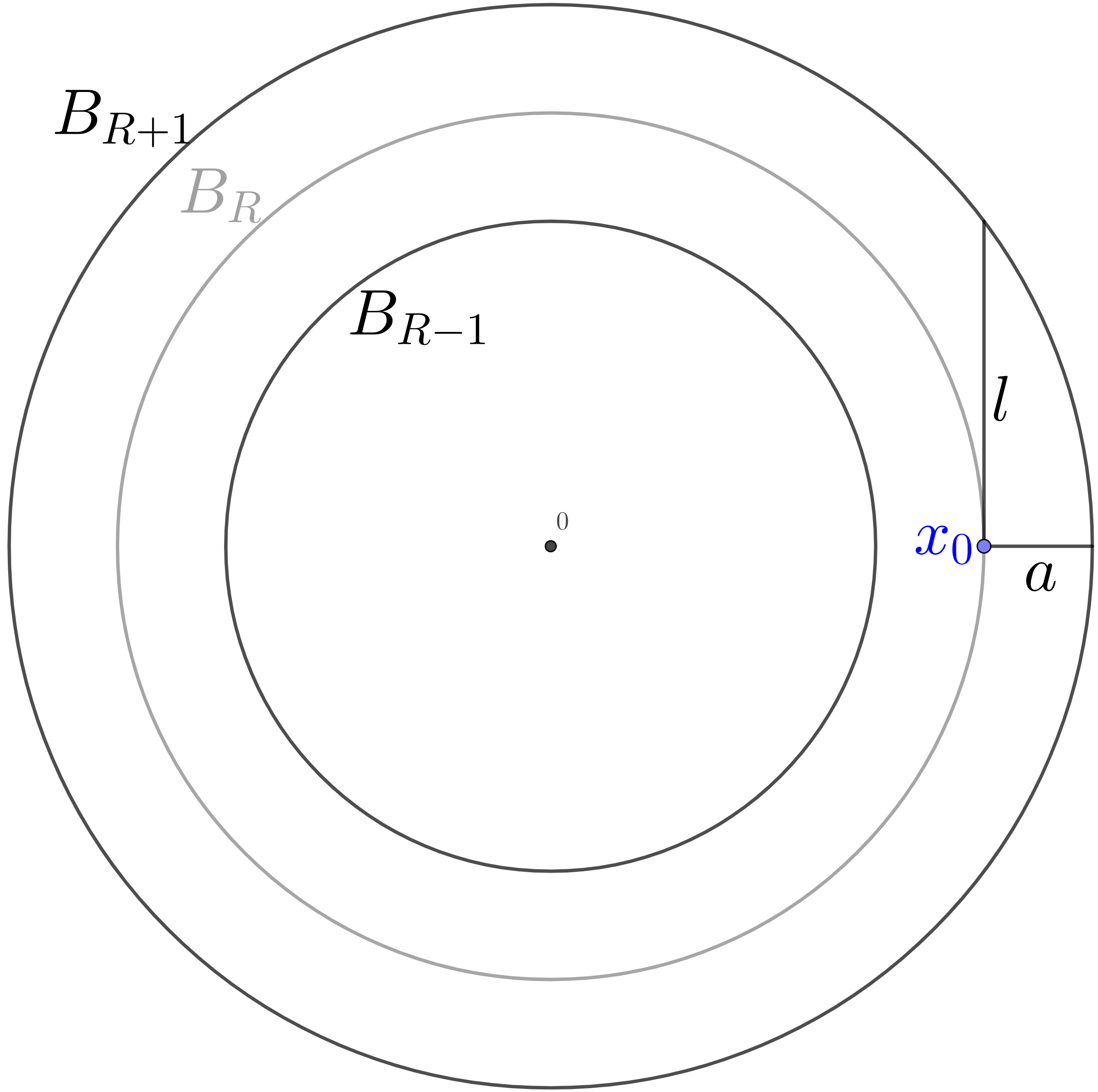}
\caption{Characteristic set for the radial counterexample.}
\label{fig:radialcounter}
\end{figure}

We note that at $x_0=(r,0)$, we have that $\phi\in C^{1,1}(x_0)\cap B(\ren)$ and $\nabla\phi(x_0)=0$. Moreover, we have $a=1$ and $R^2+l^2=(R+1)^2$, so that $l=\sqrt{2R+1}$.
We denote by $e_1=(1,0)$ and $e_2=(0,1)$. Then
\[
\begin{split}
\sup_{|y|=1} \int_0^\infty \left(\phi(x_0+\eta y)-\phi(x_0)\right) \frac{\dd \eta}{\eta^{1+2s}}&\geq \int_0^\infty \left(\phi(x_0+\eta e_2)-\phi(x_0)\right)\frac{\dd \eta}{\eta^{1+2s}}\\
&=-\int_l^\infty \frac{\dd \eta}{\eta^{1+2s}}=-\frac{1}{2s} \frac{1}{(2R+1)^s}
\end{split}
\]
and
\[
\begin{split}
\inf_{|y|=1} \int_0^\infty \left(\phi(x_0+\eta y)-\phi(x_0)\right) \frac{\dd \eta}{\eta^{1+2s}}&= \int_0^\infty \left(\phi(x_0+\eta e_1)-\phi(x_0)\right) \frac{\dd \eta}{\eta^{1+2s}}=-\int_1^\infty \frac{\dd \eta}{\eta^{1+2s}}=-\frac{1}{2s}.
\end{split}
\]
By Lemma \ref{lem:EquivalentDefIFL},
\[
\ifl \phi(x)\geq -\frac{C_s}{2s}\left(\frac{1}{(2R+1)^s}+1\right).
\]
On the other hand, we define $\Phi:\R\to \R$ by $\Phi(r):=\phi(|x|)$ when $r=|x|$ and $\Phi(-r):=\Phi(r)$, and let $r_0=|x_0|$.
\[
\begin{split}
-(-\partial_{rr}^2)^s\Phi(r_0)&= C_s\int_0^\infty \big(\Phi(r_0+\eta)+\Phi(r_0-\eta)-2\Phi(r_0)\big) \frac{\dd \eta}{\eta^{1+2s}}\\
&=C_s\int_0^\infty \left(\phi\left(x_0+\eta\frac{x_0}{|x_0|}\right)-\phi(x_0)\right) \frac{\dd \eta}{\eta^{1+2s}}+C_s\int_0^\infty \left(\phi\left(x_0-\eta\frac{x_0}{|x_0|}\right)-\phi(x_0)\right)\frac{\dd \eta}{\eta^{1+2s}}\\
&=-C_s \int_1^\infty \frac{\dd \eta}{\eta^{1+2s}}-C_s \int_1^{2R-1} \frac{\dd \eta}{\eta^{1+2s}}-C_s \int_{2R+1}^\infty \frac{\dd \eta}{\eta^{1+2s}}\\
&=-\frac{C_s}{2s}\left(2+\frac{1}{(2R+1)^{2s}}-\frac{1}{(2R-1)^{2s}}\right).
\end{split}
\]
Finally, taking $R$ big enough, we get
\[
-\ifl\phi(x_0)-(-\partial_{rr})^s\Phi(r_0)\leq -\frac{C_s}{2s}\left(1+\frac{1}{(2R+1)^{2s}}-\frac{1}{(2R-1)^{2s}}-\frac{1}{(2R+1)^s}\right)<0.
\]
Thus, the operators cannot coincide.

\subsection{Another example of smooth solutions}\label{sec:OtherExamplesOfSmoothSolutions}
We present here another example of functions for which $\ifl$ reduces to a one-dimensional fractional Laplacian. As before, this allows to produce smooth solutions of \eqref{eq:EQ}--\eqref{eq:BC}. We will adopt the notation $x=(x_1,\ldots, x_n)\in \ren$.

\begin{lemma}\label{lem:OtherExamplesSmoothSolutions}
Assume that $\Phi\in C^2_{\textup{b}}(\R)$ is nondecreasing, and let $\phi\in C^2_{\textup{b}}(\ren)$ be defined as
\[
\phi(x):=\Phi(x_1).
\] 
Then
\[
\ifl \phi(x)= -(-\partial_{x_1x_1}^2)^s \Phi(x_1).
\]
\end{lemma}

\begin{remark}
We could also take $\Phi\in C^2_{\textup{b}}(\R)$ and nonincreasing in the above result.
\end{remark}

\begin{proof}[Proof of Lemma \ref{lem:OtherExamplesSmoothSolutions}]
Note that $\nabla \phi(x)=\Phi'(x_1)e_1$ and $\Phi'(x_1)\geq0$. On one hand, if $\Phi'(x_1)=0$, then it is clear that
\[
\begin{split}
\ifl \phi(x)
&=C_s\sup_{|y|=1} \int_{0}^\infty\big(\phi(x+ \eta y)-\phi(x)\big) \frac{\dd \eta}{\eta^{1+2s}} + C_s\inf_{|y|=1} \int_{0}^\infty\big(\phi(x-\eta y)-\phi(x)\big) \frac{\dd \eta}{\eta^{1+2s}}\\
&=C_s\int_{0}^\infty\big(\phi(x+ \eta e_1)-\phi(x)\big) \frac{\dd \eta}{\eta^{1+2s}} + C_s\int_{0}^\infty\big(\phi(x-\eta e_1)-\phi(x)\big) \frac{\dd \eta}{\eta^{1+2s}}\\
&=C_s\int_{0}^\infty\big(\Phi(x_1+ \eta)+\Phi(x_1-\eta)-2\Phi(x_1)\big) \frac{\dd \eta}{\eta^{1+2s}}\\
&=-(-\partial_{x_1x_1}^2)^s \Phi(x_1).
\end{split}
\]
On the other hand, if $\Phi'(x_1)>0$, then $\zeta=\nabla\phi(x)/|\nabla\phi(x)|=e_1$ (cf. Lemma \ref{lem:EquivalentDefIFL}) and
\[
\begin{split}
\ifl \phi(x)
&=C_s\int_0^\infty\left(\phi(x+\eta e_1)+\phi(x-\eta e_1)-2\phi(x)\right)\frac{\dd \eta }{\eta ^{1+2s}}\\
&=-(-\partial_{x_1x_1}^2)^s \Phi(x_1). \qedhere
\end{split}
\]
\end{proof}

\section{Comparison and local truncation errors}\label{sec:compsmooth}

We are not able to prove comparison (neither uniqueness) for the family of viscosity solutions constructed in Section \ref{sec:existence}.
However, we are able to compare any constructed viscosity solution with any classical solution.

The argument is based on the fact that for classical solutions we can
get  full  convergence of the scheme
\eqref{eq:EQD}--\eqref{eq:BCD} (and not just  compactness and
convergence  up to a subsequence). Then we can inherit the
comparison result of the scheme to the limit solutions. 

\subsection{Comparison and convergence estimates under regularity assumptions}

\begin{proposition}\label{prop:NumMethSmooth_new}
Assume \eqref{eq:CFL} and $u_0\in BUC(\ren)$. Let $u,v\in  C^2_\textup{b}
(\ren\times[0,\infty))$ be respective classical sub- and supersolutions of
  \eqref{eq:EQ}--\eqref{eq:BC}. Then:
\begin{enumerate}[{\rm (a)}]
\item Let $U_\veps,V_\veps$ respective super- and subsolutions of the scheme \eqref{eq:EQD}--\eqref{eq:BCD}. Then, for all $T<\infty$,
\[
u+o_\veps(1) + O(\tau) \leq U_\veps\qquad\text{and}\qquad v+o_\veps(1) + O(\tau)\geq V_\veps,\qquad \text{uniformly in $\ren\times \{\tau\N\cup 0\}$.}
\]

\item Let $U_\veps$ be a solution of the scheme \eqref{eq:EQD}--\eqref{eq:BCD}. Then, for all $T<\infty$,
\[
u+o_\veps(1)+O(\tau) \leq U_\veps\leq v+o_\veps(1)+O(\tau),\qquad \text{uniformly in $\ren\times \{\tau\N\cup 0\}$.}
\]

\end{enumerate}
\end{proposition}

We immediately get:

\begin{corollary}\label{cor:NumMethSmooth}
Assume \eqref{eq:CFL}. Let $u\in  C^2_\textup{b}
(\ren\times[0,\infty))$ be a classical solution of \eqref{eq:EQ}--\eqref{eq:BC}, and $U_\veps$ be a solution of the scheme \eqref{eq:EQD}--\eqref{eq:BCD},  both with initial data $u_0$. Then, for all $T<\infty$,
\[
\max_{t_j\leq T}\|u(\cdot, t_j)- U_\veps(\cdot,t_j)\|_{C_\textup{b}(\ren)} = o_\veps(1) + O(\tau).
\]

\end{corollary}

\begin{proof}[Proof of Proposition \ref{prop:NumMethSmooth_new}]
\noindent(a) Define the local truncation error,
\begin{equation}\label{eq:LocalTruncationError}
(R_\veps)^j(x):= \frac{u(x, t_j+\tau)-u(x, t_j)}{\tau}- \ifleps[u(\cdot, t_j)](x).
\end{equation}
Clearly, since $u$ is a classical subsolution of \eqref{eq:EQ} we have (from Lemma \ref{lem:Consistency})
\[
\begin{split}
(R_\veps)^j(x)&\leq\left(\frac{u(x, t_j+\tau)-u(x, t_j)}{\tau} - \partial_t u(x,t_j)\right)- \big(\ifleps[u(\cdot, t_j)](x)- \ifl [u(\cdot, t_j)](x)\big)\\
&\leq O(\tau)+ o_\veps(1)
\end{split}
\]
with uniform bounds in $t_j$ and $x$.

Define now $e^j(x)=u(x,t_j)-(U_\veps)^j(x)=u(x,t_j)-U_\veps(x,t_j)$. By \eqref{eq:EQD} and \eqref{eq:LocalTruncationError}, we get
\[
\begin{split}
e^{j+1}(x)&=u(x,t_{j+1})-(U_\veps)^{j+1}(x)\\
%&\leq \tau(R_\veps)^j(x)+\tau \ifleps[u(\cdot, t_j)](x)+u(x,t_{j})-\big(\tau\ifleps[(U_\veps)^j](x)+(U_\veps)^j(x)\big)\\
&\leq e^{j}(x) + \tau \big( \ifleps[u(\cdot, t_j)](x) -  \ifleps[(U_\veps)^j](x) \big) + \tau (R_\veps)^j(x)\\
&=  e^j(x) (1- \tau\frac{C_s}{s \veps^{2s}}) + \tau C_s \left(\sup_{|y|=1}\int_\veps^\infty u(x+\eta y,t_j)  \frac{\dd \eta }{\eta ^{1+2s}} - \sup_{|y|=1}\int_\veps^\infty (U_\veps)^j(x+\eta y)  \frac{\dd \eta }{\eta ^{1+2s}} \right) \\
& \quad+  \tau C_s \left(\inf_{|y|=1}\int_\veps^\infty u(x+\eta y,t_j)  \frac{\dd \eta }{\eta ^{1+2s}} - \inf_{|y|=1}\int_\veps^\infty (U_\veps)^j(x+\eta y)  \frac{\dd \eta }{\eta ^{1+2s}} \right) +\tau (R_\veps)^j(x)\\
&\leq e^j(x) (1- \tau \frac{C_s}{s \veps^{2s}}) + 2\tau C_s \sup_{|y|=1}\int_\veps^\infty e^j(x+\eta y) \frac{\dd \eta }{\eta ^{1+2s}} +\tau (R_\veps)^j(x)\\
&\leq \sup_{x\in\ren}e^j(x) (1- \tau \frac{C_s}{s \veps^{2s}}) + 2\tau C_s \sup_{|y|=1}\int_\veps^\infty \sup_{x\in\ren}e^j(x+\eta y) \frac{\dd \eta }{\eta ^{1+2s}} +\tau \sup_{t_j\leq T}\sup_{x\in\ren}(R_\veps)^j(x)\\
&= \sup_{x\in\ren}e^j(x) (1- \tau \frac{C_s}{s \veps^{2s}}) + \tau \frac{C_s}{s \veps^{2s}}\sup_{x\in\ren}e^j(x)+\tau \sup_{t_j\leq T}\sup_{x\in\ren}(R_\veps)^j(x)\\
&= \sup_{x\in\ren}e^j(x)+\tau \sup_{t_j\leq T}\sup_{x\in\ren}(R_\veps)^j(x).
\end{split}
\]
I.e.,
$$
\sup_{x\in\ren}e^{j+1}(x)\leq \sup_{x\in\ren}e^j(x)+\tau \sup_{t_j\leq T}\sup_{x\in\ren}(R_\veps)^j(x).
$$
Iterating, we obtain
\[
\begin{split}
\sup_{x\in\ren}e^{j}(x)&\leq \sup_{x\in\ren}e^0(x)+j\tau\sup_{t_j\leq T}\sup_{x\in\ren}(R_\veps)^j(x)\\
&\leq  \sup_{x\in\ren}\big(u(x,0)-(U_\veps)^0(x)\big)+T\big(O(\tau)+ o_\veps(1)\big)\\
&\leq 0+O(\tau)+ o_\veps(1).
\end{split}
\]

By changing the roles of $u,U_\veps$ with $-v,-V_\veps$, we obtain the other inequality in a similar way.

\medskip
\noindent(b) Follows directly from part (a). 
\end{proof}

%%%%%%%%%%%%%%%%%%%%%%%%%%%%%%%%%%%%%%%%%%%%%%%%%%%%%%%%%%%%%%%
\subsection{Comparison for classical sub- and supersolutions}

In order to continue, we note that Proposition \ref{prop:NumMethSmooth_new} and Corollary \ref{cor:NumMethSmooth} hold exactly as before with the time interpolant $u_\veps$ replacing $U_\veps$ (cf. the proof of Corollary \ref{cor:propscheme}).

\begin{proof}[Proof of Theorem \ref{thm:comparisonsmooth}]
The proof is similar for $\overline{u},\underline{u}$, and we only
provide it for $\underline{u}$. Since $u$ is a constructed
viscosity solution in the sense of Theorem \ref{thm:ExistenceAPrioriViscosity}, by Proposition \ref{prop:SemiDiscreteCompactness} there is a sequence 
$u_{\veps_j}\in BUC(\ren\times[0,\infty))$ of
  time-interpolated  solutions of \eqref{eq:EQD}--\eqref{eq:BCD}
  with initial condition $u_0$ such that
\[
u_{\veps_j} \to  u \qquad \textup{locally uniformly in $\ren\times[0,\infty)$ as $\veps_j\to 0^+$.}
\]
Then by taking the limit as $\veps_j\to 0^+$ in Proposition \ref{prop:NumMethSmooth_new}(b), we get $\underline{u}\leq u$.
\end{proof}

\section{Global Harnack principle}\label{sec:GHP}

The proof of Theorem \ref{thm:GlobalHarnackPrinciple} is based on the relation between our problem and the smooth solutions of the fractional heat equation, the properties of smooth solutions in 1D for the fractional heat equation presented in the review Section~\ref{sec:heat},  and the comparison principle of Theorem \ref{thm:comparisonsmooth} for viscosity and classical solutions.

\begin{proof}[Proof of Theorem \ref{thm:GlobalHarnackPrinciple}]
A key point of the proof is the fact that  if $v$ is a smooth, radial, and radially nonincreasing solution of the fractional heat
equation in one dimension, then $u(x,t):=v(|x|,t)$ is a solution of
\eqref{eq:EQ}--\eqref{eq:BC}. See Theorem \ref{thm:ExistenceClassicalSolution}.

\smallskip
\noindent\textbf{1)} \emph{Upper bound.} Let $\overline{u}_0:\ren\to \R_+$ be such that (i) $\overline{u}_0(x)=(1+|x|^2)^{-\frac{1+2s}{2}}$ if $|x|\geq R+1$; (ii) $\overline{u}_0$ is radially symmetric and radially nonincreasing; (iii) $\overline{u}_0\in C^\infty_\textup{b}(\ren)$; and (iv) $u_0\leq \overline{u}_0$ in $\ren$. Consult Figure \ref{fig:upperbound}.

Moreover, let $\overline{v}_0:\R\to \R$ be defined by
$\overline{v}_0(r):=\overline{u}_0(|x|)$ with $r=|x|$ and
$\overline{v}_0(-r):=\overline{v}_0(r)$. Clearly, $\overline{v}_0\in
C^\infty_\textup{b}(\R)$ is radially symmetric and radially
nonincreasing. Let $\overline{v}$ be the corresponding solution of the
fractional heat equation \eqref{eq:FHEQ}--\eqref{eq:FHBC} and define
$\overline{u}(x,t)=\overline{v}(|x|,t)$. By %the proof of
 Theorem \ref{thm:ExistenceClassicalSolution}, $\overline{u}\in C^\infty_\textup{b}(\ren\times[0,\infty))$ is a classical solution of \eqref{eq:EQ}--\eqref{eq:BC}. Moreover, $\overline{u}$ is radial and radially nonincreasing. Since  $\overline{u}_0(x)=(1+|x|^2)^{-\frac{1+2s}{2}}$ if $|x|\geq R+1$, then $\overline{v}_0(r)=(1+|r|^2)^{-\frac{1+2s}{2}}$ if $|r|\geq R+1$, so that, by Lemma \ref{lem:finetailHE}, for all $t>\tau$ we have
\[
\overline{u}(x,t)=\overline{v}(|x|,t) \leq k_2\|\overline{v}_0\|_{L^1(\R)}P_s(|x|,t)\leq C_1 \frac{t}{(t^{\frac{1}{s}}+|x|^2)^{\frac{1+2s}{2}}}
\]
Finally, since $\overline{u}\in C^\infty_\textup{b}(\ren\times[0,\infty))$ is a classical solution of \eqref{eq:EQ}--\eqref{eq:BC} and $u_0\leq \overline{u}_0$ we have, by Theorem \ref{thm:comparisonsmooth}, that $u(x,t) \leq \overline{u}(x,t)$.

\begin{figure}[h!]
\center
\includegraphics[width=0.5\textwidth]{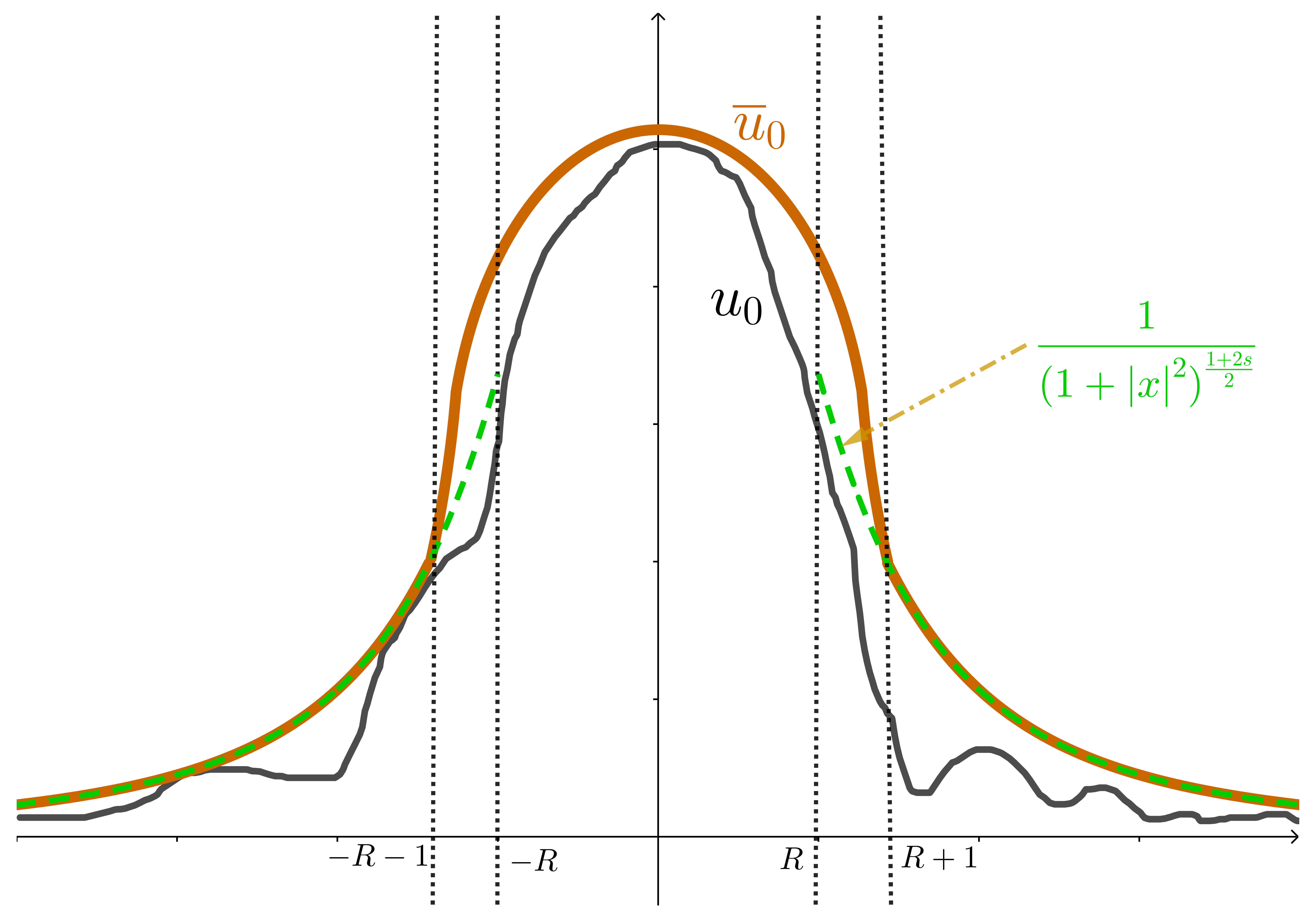}
\caption{Upper bound for $u_0$ in the proof of Theorem \ref{thm:GlobalHarnackPrinciple}.}
\label{fig:upperbound}
\end{figure}

\medskip
\noindent\textbf{2)} \emph{Lower bound.}  Without loss of generality,
assume $ u_0(0)=\sup_{x\in \ren} u_0(x)$>0. By continuity of $u_0$,
there exists $R_0>0$ such that $u_0(x)\geq u_0(0)/2$ for all $x\in
B_{R_0}(0)$. Consider e.g. the  scaled standard mollifier
\[
\underline{u}_0(x)=\frac{u_0(0)}{2}\textup{e}^{1-\frac{R_0^2}{(R_0^2-|x|^2)_+}}
\]
Clearly, (i) $\underline{u}_0(x)=0\leq (1+|x|^2)^{-\frac{1+2s}{2}}$ if $|x|\geq R_0$; (ii) $\underline{u}_0$ is radially symmetric and radially nonincreasing; (iii) $\underline{u}_0\in C^\infty_\textup{b}(\ren)$; and (iv) $u_0\geq \underline{u}_0$ in $\ren$ since
\[\underline{u}_0(x) \leq \underline{u}_0(0) = \frac{u_0(0)}{2}\textup{e}^{1-\frac{R_0^2}{(R_0^2-0)_+}}=  \frac{u_0(0)}{2} \leq u_0(x) \quad \text{for all $x\in B_{R_0}(0)$}
\]
and $\underline{u}_0(x)=0\leq u_0$ for $x\in\ren \setminus B_{R_0}(0)$ (see Figure \ref{fig:lowerbound}). From here, the proof follows as in Step 1) by using the lower bound in Lemma \ref{lem:finetailHE}.

\begin{figure}[h!]
\center
\includegraphics[width=0.5\textwidth]{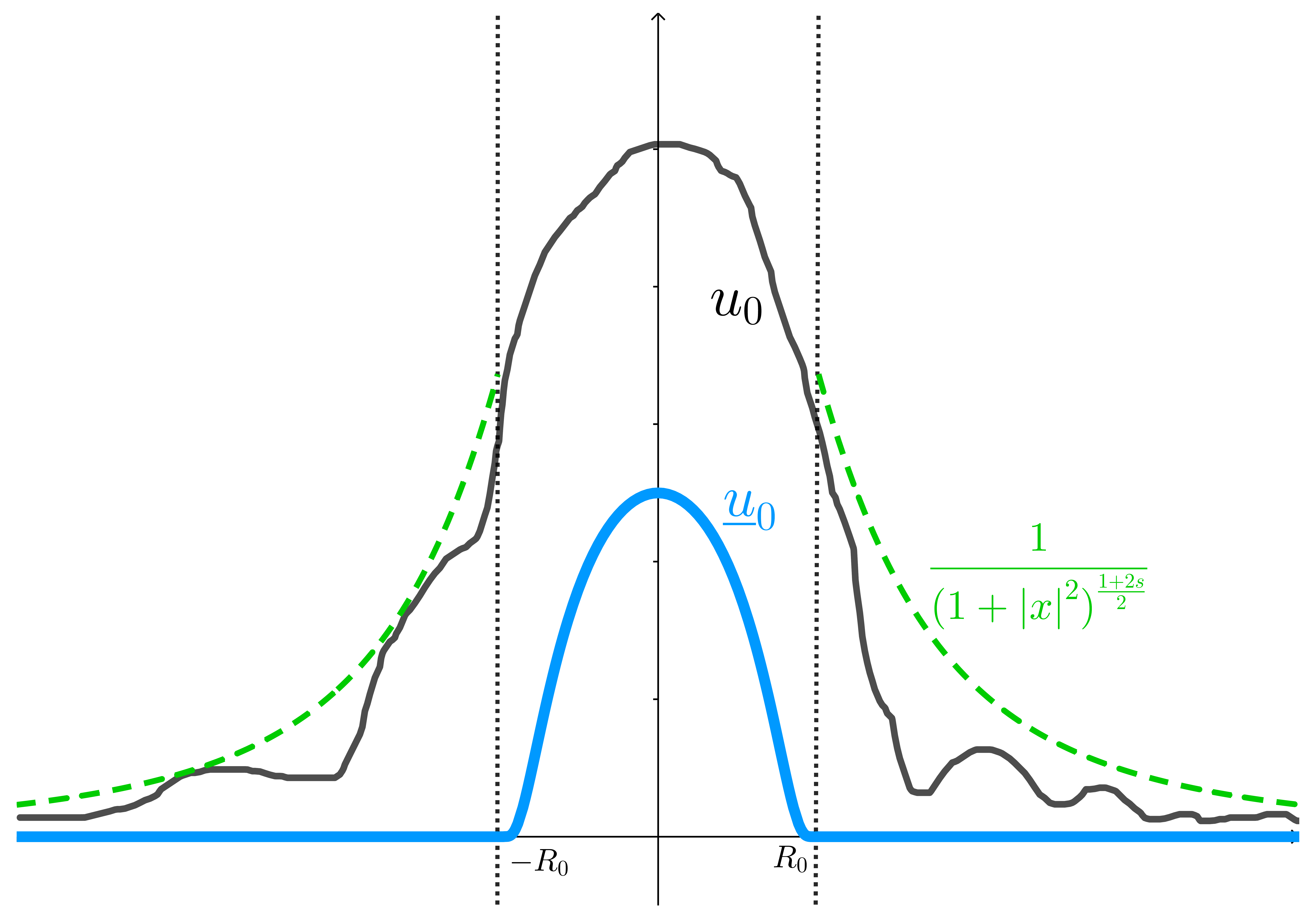}
\caption{Lower bound for $u_0$ in the proof of Theorem \ref{thm:GlobalHarnackPrinciple}.}
\label{fig:lowerbound}
\end{figure}
\end{proof}

\section{Extensions and open problems}\label{sec:extensions}

%\subsection{Extensions and open problems}

$\bullet$ There is an important open problem concerning the uniqueness and general comparison principle of viscosity solutions, either defined in our way or another suitable way that includes existence. For the moment we know that the following two classes of $BUC$ viscosity solutions are unique: (i) radial radially nonincreasing solutions and (ii) monotone solutions evolving in one dimension only. Uniqueness in these cases follows by comparison with classical solutions. The problem is also open for elliptic equations of the same type, cf. \cite{BjCaFi12}. 

$\bullet$  A main question that we deal with here is: how different is the theory and its results from the linear case (fractional heat equation)? The answer seems to be that they are quite different if $n\geq2$, since then the  infinity fractional Laplacian is a heavily nonlinear operator.

$\bullet$ It is not clear whether for $n\geq2$ the solutions evolve in time towards a radial profile (as in the local case, see below) or preserve a certain distortion. This is an interesting open problem to which we give a partial answer in our Section \ref{sec:GHP} with the global Harnack principle. 
 In Figure \ref{fig:6timeslevelsets} (obtained with a rigorous finite difference scheme taken from the companion paper \cite{DTEnJaVa22b}) the distortion present in the initial datum can still be observed for all the computed times. 
 
 \begin{figure}[h!]
\center
\includegraphics[width=0.5\textwidth]{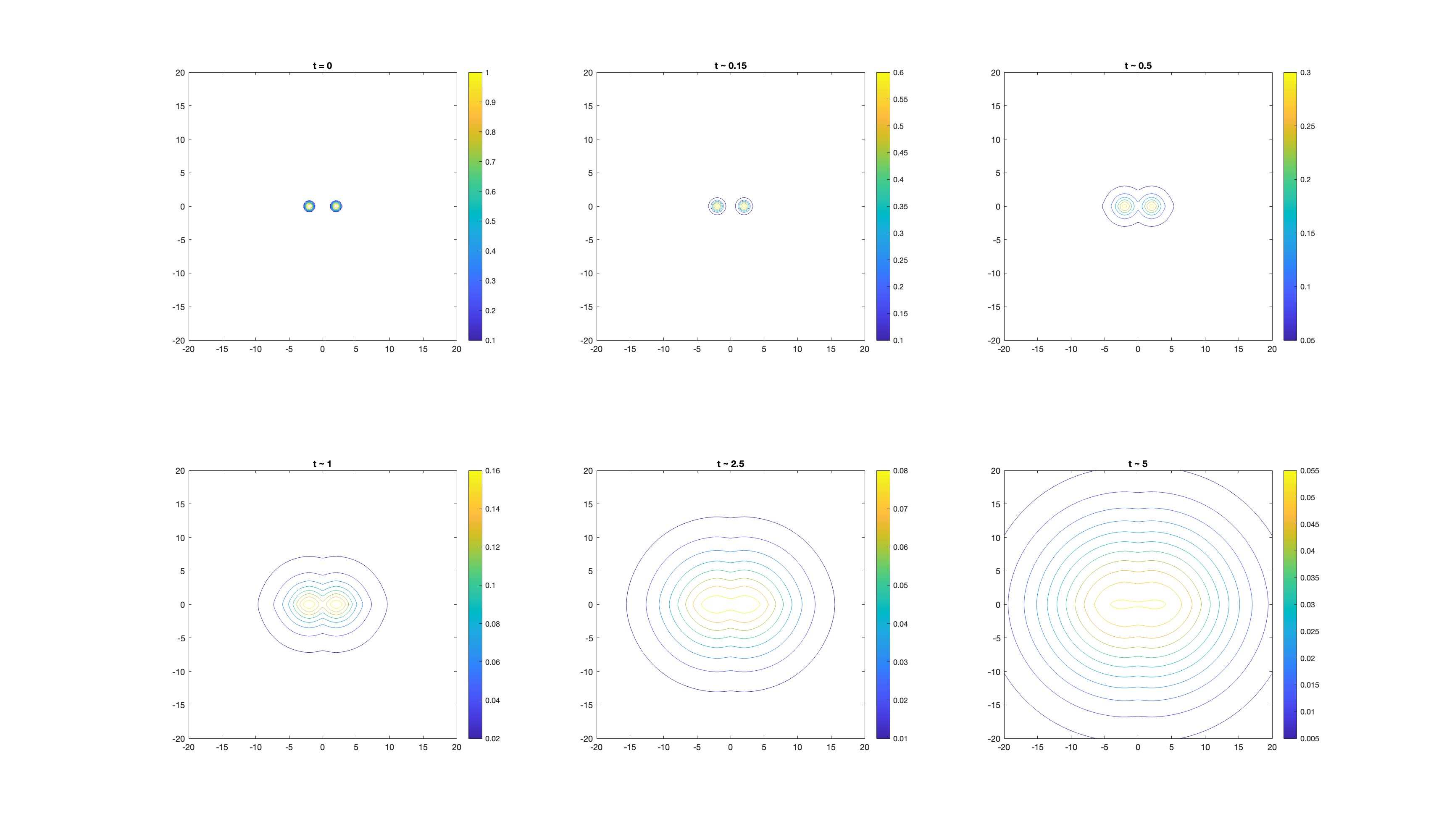}
\caption{Evolution of the level sets of the solution.}
\label{fig:6timeslevelsets}
\end{figure}

$\bullet$ In particular, the evolution equation for the local version posed in the whole space has been studied  by Portilheiro et al. in \cite{MR2915863, PortVaz2013}. Then, there is a fine asymptotic behaviour as $t\to \infty$ that implies a sharp convergence rate to radiality. The Aleksandrov Principle is a main ingredient in the proof. On the condition that the Aleksandrov Principle is true for some class of solutions of our Cauchy problem, we could also obtain a similar sharp asymptotic behaviour as $t\to \infty$ for such solutions. Such discussion is not included here.

$\bullet$ Large part of the concepts and results of this paper can be applied to the more general equation $\partial_tu =\ifl u+ f(x,t)$. In particular, this could be applied to the stationary equation $\ifl u = f(x)$, thus relating the present results to the results of \cite{BjCaFi12}.

$\bullet$ We end the discussion by including an example demonstrating that $\ifl$ could indeed be pointwise  discontinuous. Consider $\Phi\in C_{\textup{b}}^2(\R)$ satisfying $\Phi(x_1)=\Phi(-x_1)$ and strictly decreasing for $x_1\geq0$. As in Lemma \ref{lem:OtherExamplesSmoothSolutions}, we define $\phi(x):=\Phi(x_1)$ (see Figure \ref{fig:counter}) where, for the sake of simplicity, $x=(x_1,x_2)\in \R^2$. 

 \begin{figure}[h!]
\center
\includegraphics[width=0.5\textwidth]{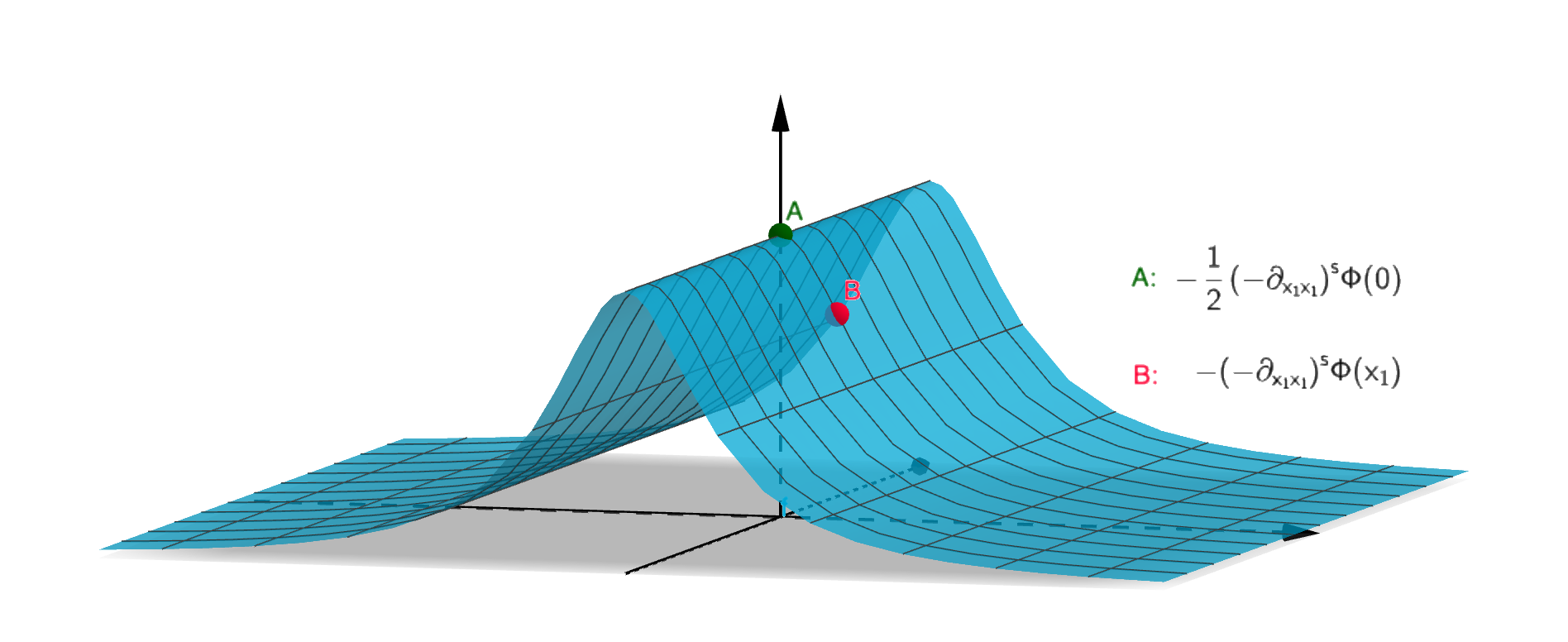}
\caption{Example of discontinuity of the operator $\ifl$.}
\label{fig:counter}
\end{figure}

On one hand, when $x_1\neq0$, we have $\zeta=\pm e_1$ (cf. Lemma \ref{lem:EquivalentDefIFL}) which yields $\ifl \phi(x_1,x_2)=-(-\partial_{x_1x_1}^2)^s\Phi(x_1)$. On the other hand, when $x_1=0$, $\nabla\phi(0,x_2)=0$ and by construction, 
$$
\inf_{|y|=1} \int_{0}^\infty\big(\phi(0,x_2)-\eta y)-\phi(0,x_2)\big) \frac{\dd \eta}{\eta^{1+2s}}= \int_{0}^\infty\big(\Phi(x_1-\eta)-\Phi(x_1)\big) \frac{\dd \eta}{\eta^{1+2s}}.
$$
Since $x_1=0$ is a maximum point and $\phi(0,x_2)=\phi((0,x_2)+\eta e_2)=\Phi(x_1)$, 
$$
0\geq \sup_{|y|=1} \int_{0}^\infty\big(\phi((0,x_2)+ \eta y)-\phi(0,x_2)\big) \frac{\dd \eta}{\eta^{1+2s}}\geq  \int_{0}^\infty\big(\phi((0,x_2)+ \eta e_2)-\phi(0,x_2)\big) \frac{\dd \eta}{\eta^{1+2s}}= 0.
$$
We then conclude by Lemma \ref{lem:EquivalentDefIFL} and symmetry of $\Phi$ that
\[
\begin{split}
\ifl \phi(0,x_2)
&=C_s\int_{0}^\infty\big(\Phi(- \eta )-\Phi(0)\big) \frac{\dd \eta}{\eta^{1+2s}}= \frac{1}{2}C_s\int_{0}^\infty\big(\Phi(\eta)+\Phi(- \eta )-2\Phi(0)\big) \frac{\dd \eta}{\eta^{1+2s}}\\
&=-\frac{1}{2}(-\partial_{x_1x_1}^2)^s\Phi(0).\\
\end{split}
\]

Hence,
$$
\ifl\phi(x_1,x_2)=
\begin{cases}
-(-\partial_{x_1x_1}^2)^s\Phi(x_1),&\qquad\text{if}\qquad x_1\neq0,\\
-\frac{1}{2}(-\partial_{x_1x_1}^2)^s\Phi(x_1),&\qquad\text{if}\qquad x_1=0.
\end{cases}
$$

%%%%%%%%%%%%%%%%%%%%%%%%%%%%%%%%%%%%%%%%%%%%%%%%%%%%%%%%%%%%%%%%%%%%%%
\section*{Acknowledgements}

\  F. del Teso was supported by the Spanish Government through PGC2018-094522-B-I00, RYC2020-029589-I, and PID2021-127105NB-I00 funded by the MICIN/AEI.

E. R. Jakobsen received funding from the Research Council of Norway under the Toppforsk (research excellence) grant agreement no. 250070 ``Waves and Nonlinear Phenomena (WaNP)''.

J. Endal received funding from the European Union's Horizon 2020 research and innovation programme under the Marie Sk{\l}odowska-Curie grant agreement no. 839749 ``Novel techniques for quantitative behavior of convection-diffusion equations (techFRONT)'', and from the Research Council of Norway under the MSCA-TOPP-UT grant agreement no. 312021.

The work of J. L. V\'azquez was funded by grant PGC2018-098440-B-I00 and PID2021-127105NB-I00 from the Spanish Government.  He is an Honorary Professor at Univ.\ Complutense de Madrid.

Part of this material is based upon work supported by the Swedish Research Council under grant no. 2016-06596 while the authors FdT and JLV were in residence at Institut Mittag-Leffler in Djursholm, Sweden, during the research program ``Geometric Aspects of Nonlinear Partial Differential Equations'', fall of 2022.

%%%%%%%%%%%%%%%%%%%%%%%%%%%%%%%%%%%%%%%%%%%%%%%%%%%%%%%%%%%%%%%%%%%%%%
\printbibliography

\end{document}